\documentclass[twoside]{article}
\usepackage[cp1251]{inputenc}
\usepackage[T2A]{fontenc}
\usepackage{array}
\usepackage{amssymb}
\usepackage{amsmath}
\usepackage{amsthm}
\usepackage{latexsym}
\usepackage{indentfirst}
\usepackage{bm}
\usepackage{enumerate}
\usepackage[dvips]{graphicx}
\usepackage{epsf}
\usepackage{wrapfig}
\usepackage{euscript}
\usepackage{indentfirst}
\usepackage[english,russian]{babel}
\usepackage{textcomp}
\newtheorem{theorem}{Theorem}
\newtheorem{lemma}{Lemma}
\newtheorem{definition}{Definition}
\newtheorem{proposition}{Proposition}

\newtheorem{corollary}{Corollary}
\newcommand{\co}{{\hspace{0.25mm}:\hspace{0.25mm}}}
\begin{document}
{\selectlanguage{english}
\binoppenalty = 10000 %
\relpenalty   = 10000 %

\pagestyle{headings} \makeatletter
\renewcommand{\@evenhead}{\raisebox{0pt}[\headheight][0pt]{\vbox{\hbox to\textwidth{\thepage\hfill \strut {\small Grigory. K. Olkhovikov}}\hrule}}}
\renewcommand{\@oddhead}{\raisebox{0pt}[\headheight][0pt]{\vbox{\hbox to\textwidth{{Completeness for explicit jstit logic}\hfill \strut\thepage}\hrule}}}
\makeatother

\title{Stit logic of justification announcements: a completeness result}
\author{Grigory K. Olkhovikov\\Ruhr University, Bochum \\
Department of Philosophy II; Room NAFO 02/299 \\
Universit\"atsstra{\ss}e 150 \\
D--44780 Bochum, Germany \\
Tel.: +4915123707030\\
Email: grigory.olkhovikov@\{rub.de, gmail.com\}}
\date{}
\maketitle
\begin{quote}
\textbf{Abstract}. We present a completeness result for a logical
system which combines stit logic and justification logic in order
to represent proving activity of the agents. This logic is
interpreted over the semantics introduced in \cite{OLWA}. We
define a Hilbert-style axiomatic system for this logic and show
that this system is strongly complete relative to the intended
semantics.
\end{quote}

\begin{quote}
stit logic, justification logic, completeness, compactness
\end{quote}

\section{Introduction}
Stit logic of justification announcements (JA-STIT) is a formalism
for reasoning about proving activity of agents which combines
expressive powers of stit logic (see e.g. \cite{belnap2001facing})
with those of justification logic (see e.g. \cite{ArtemovN05}).
The two latter logics provide for the pure agency side and the
pure proof ontology side of the proving activity, respectively, so
that it is assumed that doing something is in effect seeing to it
that something is the case, and that every actual proof can be
understood as a realization of some proof polynomial from
justification logic. The only missing element in this picture is
then the link between the two components, i.e. how agents can see
to it that a proof is realized. Such a realization may come in
different forms, researchers may, for example, exchange emails or
put the proofs they have found on a common whiteboard. In stit
logic of justification announcements this rather common situation
is idealized in that only public proving activity of agents is
taken into account. In other words, taking up the whiteboard
metaphor, the agents in question can only participate in proving
activity by putting their proofs on the common whiteboard for
everyone to see, and not by sending one another private messages
or scribbling in their private notebooks. Therefore, the only type
of communicative actions within this idealized community turns out
to be a variant of \emph{public announcement} of proof
polynomials.\footnote{Even though this type of public announcement
actions plays a central role in our logic, finding any sort of
meaningful connection to the well-known public announcement logic
(PAL) looks like a non-trivial matter. One obvious reason for this
is the difference between the underlying action logics (stit logic
in the case at hand vs. dynamic logic of PAL). Another reason is
that we are studying public announcements of a special type of
objects (i.e. proof polynomials) in a multi-agent setting, whereas
in PAL sentence announcements are studied, and these sentence
announcements are not tied to a particular agent. Moreover, in
JA-STIT public announcements are not reducible to static formulas
and are actually intended to be that way.} This idealization lends
the medium of public announcement, i.e. the metaphorical community
whiteboard, the status of the only interface between the agentive
efforts of the community and the abstract realm of proofs. Proof
polynomials may end up being \emph{presented} on the whiteboard,
and the agents may \emph{see to it} that this or that particular
proof is presented there. The whiteboard itself is also idealized
in that we assume that there is always enough space on it to put
up another proof, and that every proof, once on the whiteboard,
remains there forever.

The language of stit logic of justification announcements then
combines the full sets of justification and stit modalities with a
new modality $Et$ which says that the proof polynomial $t$ is
presented to the community (or, to continue with the whiteboard
metaphor, that $t$ is put on the community whiteboard). In this
way arises a non-trivial and expressively rich logic, and the main
purpose of the present paper is to provide a strongly complete
axiomatization for this logic.

The layout of the rest of the paper is as follows. In Section
\ref{basic} we define the language and the semantics of the logic
at hand. We also briefly characterize its relations with other
formalisms combining the resources of justification logic and stit
logic, studied in the earlier publications, namely, in \cite{OLWA}
and \cite{OLWA2}. We mention that the finite model property fails
for the stit logic of justification announcements in a rather
strong form, and show that the language of JA-STIT is expressive
enough to distinguish between the full class of justification stit
models and the class of justification stit models based on
discrete time.

The system of axioms for JA-STIT is then presented in Section
\ref{axioms}. We immediately show this system to be sound w.r.t.
the semantics introduced in Section \ref{basic}, and deduce some
theorems in the system.

Section \ref{canonicalmodel} then contains the bulk of technical
work necessary for the proof of completeness of the presented
axiomatization w.r.t. the class of normal jstit models. It gives a
stepwise construction and adequacy check for all the numerous
components of the canonical model and ends with a proof of a truth
lemma. Section \ref{main} then gives a concise proof of the
completeness result and draws some quick corollaries including the
compactness property.

Then follows Section \ref{conclusion}, giving some conclusions and
drafting directions for future work.

In what follows we will be assuming, due to space limitations, a
basic acquaintance with both stit logic and justification logic.
We recommend to peruse \cite[Ch. 2]{horty2001agency} for a quick
introduction to the basics of stit logic, and
\cite{sep-logic-justification} for the same w.r.t. justification
logic.

\section{Basic definitions and notation}\label{basic}

\subsection{Preliminaries}

We fix some preliminaries. First, we choose a finite set $Ag$
disjoint from all the other sets to be defined below. Individual
agents from this set will be denoted by letters $i$ and $j$. Then
we fix countably infinite sets $PVar$ of proof variables (denoted
by $x,y,z,w,u$) and $PConst$ of proof constants (denoted by
$a,b,c,d$). When needed, subscripts and superscripts will be used
with the above notations or any other notations to be introduced
in this paper. Set $Pol$ of proof polynomials is then defined by
the following BNF:
$$
t := x \mid c \mid s + t \mid s \times t \mid !t,
$$
with $x \in PVar$, $c \in PConst$, and $s,t$ ranging over elements
of $Pol$. In the above definition, $+$ stands for the \emph{sum}
of proofs, $\times$ denotes \emph{application} of its left
argument to the right one, and $!$ denotes the so-called
\emph{proof-checker}, so that $!t$ checks the correctness of proof
$t$.

In order to define the set $Form$ of formulas, we fix a countably
infinite set $Var$ of propositional variables to be denoted by
letters $p,q,r,s$. Formulas themselves will be denoted by letters
$A,B,C,D$, and the definition of $Form$ is supplied by the
following BNF:
\begin{align*}
A := p \mid A \wedge B \mid \neg A \mid [j]A \mid \Box A \mid t\co
A \mid KA \mid Et,
\end{align*}
with $p \in Var$, $j \in Ag$ and $t \in Pol$.

It is clear from the above definition of $Form$ that we are
considering a version of modal propositional language. As for the
informal interpretations of modalities, $[j]A$ is the so-called
cstit action modality and $\Box$ is the historical necessity
modality; both modailities are borrowed from stit logic. The next
two modailities, $KA$ and $t\co A$, come from justification logic
and the latter is interpreted as ``$t$ proves $A$'', whereas the
former is the strong epistemic modality ``$A$ is known''.

We assume $\Diamond$ as notation for the dual modality of $\Box$.
As usual, $\omega$ will denote the set of natural numbers
including $0$, ordered in the natural way.

\subsection{Semantics}

For the language at hand, we assume the following semantics. A
\emph{justification stit} (or jstit, for short) \emph{model} is a
structure
$$
\mathcal{M} = \langle Tree, \unlhd, Choice, Act, R, R_e,
\mathcal{E}, V\rangle
$$
such that:
\begin{enumerate}
\item $Tree$ is a non-empty set. Elements of $Tree$ are called
\emph{moments}.

\item $\unlhd$ is a partial order on $Tree$ for which a temporal
interpretation is assumed. We will also freely use notations like
$\unrhd$, $\lhd$, and $\rhd$ to denote the inverse relation and
the irreflexive companions.\footnote{A more common notation $\leq$
is not convenient for us since we also widely use $\leq$ in this
paper to denote the natural order relation between elements of
$\omega$.}

\item $Hist(\mathcal{M})$ is a set of maximal $\unlhd$-chains in
$Tree$. Since $Hist(\mathcal{M})$ is completely determined by
$Tree$ and $\unlhd$, it is not included into the model structure
as a separate component. Elements of $Hist(\mathcal{M})$ are
called \emph{histories}. The set of histories containing a given
moment $m$ will be denoted $H_m$. The following set:
$$
MH(\mathcal{M}) = \{ (m,h)\mid m \in Tree,\, h \in H_m \},
$$
called the set of \emph{moment-history pairs}, will be used to
evaluate the elements of $Form$.

\item $Choice$ is a function mapping $Tree \times Agent$ into
$2^{2^{Hist}}$ in such a way that for any given $j \in Agent$ and
$m \in Tree$ we have as $Choice(m,j)$ (to be denoted as
$Choice^m_j$ below) a partition of $H_m$. For a given $h \in H_m$
we will denote by $Choice^m_j(h)$ the element of partition
$Choice^m_j$ containing $h$.

\item $Act$ is a function mapping $MH(\mathcal{M})$ into
$2^{Pol}$.

\item $R$ and $R_e$ are two pre-order on $Tree$ giving two
versions of epistemic accessibility relation. They are assumed to
be connected by inclusion $R \subseteq R_e$.

\item $\mathcal{E}$ is a function mapping $Tree \times Pol$ into
$2^{Form}$.

\item $V$ is the evaluation function, mapping the set $Var$ into
$2^{MH(\mathcal{M})}$.
\end{enumerate}

However, not all structures of the above described type are
admitted as jstit models. A number of additional restrictions
needs to be satisfied. More precisely, we assume satisfaction of
the following constraints:

\begin{enumerate}
\item Historical connection:
$$
(\forall m,m_1 \in Tree)(\exists m_2 \in Tree)(m_2 \unlhd m \wedge
m_2 \unlhd m_1).
$$

\item No backward branching:
$$
(\forall m,m_1,m_2 \in Tree)((m_1 \unlhd m \wedge m_2 \unlhd m)
\to (m_1 \unlhd m_2 \vee m_2 \unlhd m_1)).
$$

\item No choice between undivided histories:
$$
(\forall m,m' \in Tree)(\forall h,h' \in H_m)(m \lhd m' \wedge m'
\in h \cap h' \to Choice^m_j(h) = Choice^m_j(h'))
$$
for every $j \in Agent$.

\item Independence of agents:
$$
(\forall m\in Tree)(\forall f:Ag \to 2^{H_m})((\forall j \in
Ag)(f(j) \in Choice^m_j) \Rightarrow \bigcap_{j \in Ag}f(j) \neq
\emptyset).
$$

\item Monotonicity of evidence:
$$
(\forall t \in Pol)(\forall m,m' \in Tree)(R_e(m,m') \Rightarrow
\mathcal{E}(m,t) \subseteq \mathcal{E}(m',t)).
$$

\item Evidence closure properties. For arbitrary $m \in Tree$,
$s,t \in Pol$ and $A, B \in Form$ it is assumed that:
\begin{enumerate}
\item $A \to B \in \mathcal{E}(m,s) \wedge A \in \mathcal{E}(m,t)
\Rightarrow B \in \mathcal{E}(m,s\times t)$;

\item $\mathcal{E}(m,s) \cup \mathcal{E}(m,t) \subseteq
\mathcal{E}(m,s + t)$.
\item $A \in \mathcal{E}(m,t) \Rightarrow
t:A \in \mathcal{E}(m,!t)$;
\end{enumerate}

\item Expansion of presented proofs:
$$
(\forall m,m' \in Tree)(m' \lhd m \Rightarrow \forall h \in H_m
(Act(m',h) \subseteq Act(m,h))).
$$

\item No new proofs guaranteed:
$$
(\forall m \in Tree)(\bigcap_{h \in H_m}(Act(m,h)) \subseteq
\bigcup_{m' \lhd m, h \in H_m}(Act(m',h))).
$$

\item Presenting a new proof makes histories divide:
$$
(\forall m \in Tree)(\forall h,h' \in H_m)((\exists m' \rhd m)(m'
\in h \cap h') \Rightarrow (Act(m,h) = Act(m,h'))).
$$

\item Future always matters:
$$
\unlhd \subseteq R.
$$

\item Presented proofs are epistemically transparent:
$$
(\forall m,m' \in Tree)(R_e(m,m') \Rightarrow (\bigcap_{h \in
H_m}(Act(m,h)) \subseteq \bigcap_{h' \in H_{m'}}(Act(m',h')))).
$$
\end{enumerate}

We offer some intuitive explanation for the above defined notion
of jstit model. Jstit models were introduced in \cite{OLWA} for
the logics based on the combination of stit and justification
modalities. Due to space limitations, we only explain the
intuitions behind jstit models very briefly, and we urge the
reader to consult \cite[Section 3]{OLWA} for a more comprehensive
explanations, whenever needed.

The components like $Tree$, $\unlhd$, $Choice$ and $V$ are
inherited from stit logic, whereas $R$, $R_e$, and $\mathcal{E}$
come from justification logic. The only new component is the
function $Act$, which gives out, to take up the whiteboard
metaphor, the current state of this whiteboard at any given moment
under any given history. When interpreting $Act$, we draw on the
classical stit distinction between dynamic (agentive) and static
(moment-determinate) entities, assuming that the presence of a
given proof polynomial $t$ on the community whiteboard only
becomes an accomplished fact at $m$ when $t$ is present in
$Act(m,h)$ for \emph{every} $h \in H_m$. On the other hand, if $t$
is in $Act(m,h)$ only for \emph{some} $h \in H_m$ this means that
$t$ is rather in a dynamic state of \emph{being presented}, rather
than being present, to the community.

The numbered list of semantical constraints above then just builds
on these intuitions. Constraints $1$--$4$ are borrowed from stit
logic, constraints $5$ and $6$ are inherited from justification
logic. Constraint $7$ just says that nothing gets erased from the
whiteboard, constraint $8$ says a new proof cannot spring into
existence as a static (i.e. moment-determinate) feature of the
environment out of nothing, but rather has to come as a result (or
a by-product) of a previous activity. Constraint $9$ is just a
corollary to constraint $3$ in the richer environment of jstit
models, constraint $10$ says that the possible future of the given
moment is always epistemically relevant in this moment, and
constraint $11$ says that the community immediately knows
everything that has firmly made its way onto the whiteboard.

For the members of $Form$, we will assume the following
inductively defined satisfaction relation. For every jstit model
$\mathcal{M} = \langle Tree, \unlhd, Choice, Act, R,R_e,
\mathcal{E}, V\rangle$ and for every $(m,h) \in MH(\mathcal{M})$
we stipulate that:

\begin{align*}
&\mathcal{M}, m, h \models p \Leftrightarrow (m,h) \in
V(p);\\
&\mathcal{M}, m, h \models [j]A \Leftrightarrow (\forall h'
\in Choice^m_j(h))(\mathcal{M}, m, h' \models A);\\
&\mathcal{M}, m, h \models \Box A \Leftrightarrow (\forall h'
\in H_m)(\mathcal{M}, m, h' \models A);\\
&\mathcal{M}, m, h \models KA \Leftrightarrow \forall m'\forall
h'(R(m,m') \& h' \in H_{m'} \Rightarrow \mathcal{M}, m', h'
\models A);\\
&\mathcal{M}, m, h \models t\co A \Leftrightarrow A \in
\mathcal{E}(m,t) \& (\forall m' \in Tree)(R_e(m,m') \& h' \in
H_{m'} \Rightarrow \mathcal{M}, m', h'
\models A);\\
&\mathcal{M}, m, h \models Et \Leftrightarrow t \in Act(m,h).
\end{align*}
In the above clauses we assume that $p \in Var$; we also assume
standard clauses for the Boolean connectives.  We further assume
standard definitions for satisfiability and
 validity of formulas and sets of formulas in the presented
semantics.

One can in principle simplify the above semantics by introducing
the additional constraint that $R_e \subseteq R$. This leads to a
collapse of the two epistemic accessibility relation into one.
Therefore, we will call jstit models satisfying $R_e \subseteq R$
\emph{unirelational jstit models}. It is known that such a
simplification in the context of pure justification logic does not
affect the set of theorems (see, e.g. \cite{fittingreport} and
\cite[Comment 6.5]{ArtemovN05}), and we will show that this is
also the case within the more expressive environment of JA-STIT.
In fact, the canonical model to be constructed in our completeness
is unirelational, therefore, we offer some comments as to the
simplifications of semantics available in the unirelational
setting.

We observe that one can equivalently define a unirelational jstit
model as a structure $\mathcal{M} = \langle Tree, \unlhd, Choice,
Act, R, \mathcal{E}, V\rangle$ satisfying all the constraints for
the jstit models, except that in the numbered constraints one
substitutes $R$ for $R_e$. Also, in the context of unirelational
jstit models, it is possible to simplify the satisfation clause
for $t\co A$ as follows:
$$
    \mathcal{M}, m, h \models t\co A \Leftrightarrow A \in
\mathcal{E}(m,t) \& \mathcal{M}, m, h \models KA.
$$
Before we move on, we briefly clarify the relation of JA-STIT to
other logics based on the combination of justification and stit
modalities to be found in the existing literature. Firstly,
JA-STIT is a fragment of the logic introduced in \cite{OLWA2}
under the name `logic of $E$-notions'. The difference is that in
the logic of $E$-notions an implicit version of $Et$-modality is
also present. This implicit version comes in the format $EA$,
where $A \in Form$ and has the meaning that \emph{some} proof of
$A$ is presented to the community. The satisfaction clause for
this additional modality looks as follows:
$$
\mathcal{M}, m, h \models EA \Leftrightarrow (\exists t\in Pol)(t
\in Act(m,h)  \& \mathcal{M}, m, h \models t\co A).
$$
It is pretty obvious that $EA$ is not definable using expressive
powers of JA-STIT, so that JA-STIT is a proper fragment of the
logic of $E$-notions.

Another natural logic featuring the full set of justification and
stit modalities is the basic jstit logic introduced in \cite{OLWA}
and further explored in \cite{OLWA2}. This logic is also
interpreted over the class of jstit models which facilitates the
comparison. In basic jstit logic justification and stit modalities
are augmented with the following set of four modalities
representing different modes of proving activity:
\begin{center}
 \begin{tabular}{|c| c|}
 \hline
 Notation & Informal interpretation\\
 \hline
 $Prove(j, A)$ & Agent $j$ proves $A$\\
$Prove(j, t, A)$ & Agent $j$ proves $A$ by $t$ \\
$Proven(A)$ & $A$ has been proven\\
$Proven(t, A)$ & $A$ has been proven by $t$\\
 \hline
\end{tabular}
\end{center}
In the above table, $A \in Form$, $t \in Pol$ and $j \in Ag$ are
designating the type and arrangement of arguments for the listed
modalities. It turns out that two of these four modalities, namely
$Prove(j,t,A)$ and $Proven(t,A)$ are definable in JA-STIT. These
modalities are interpreted by the following satisfaction clauses:
\begin{align*}
 &\mathcal{M}, m, h \models Prove(j, t, A) \Leftrightarrow
(\forall h' \in Choice^m_j(h))(t \in Act(m,h') \& \mathcal{M},m,h
\models t\co A)
\&\\
&\qquad\qquad\qquad\qquad\qquad \&(\exists h'' \in H_m) (t \notin
Act(m,h''));\\
&\mathcal{M}, m, h \models Proven(t, A) \Leftrightarrow (\forall
h' \in H_m) (t \in Act(m,h') \& \mathcal{M}, m, h \models t\co A)
\end{align*}
It is easy to see then that these modalities can be defined within
JA-STIT as follows:
$$
Prove(j,t,A) =_{df} [j]Et \wedge \Diamond\neg Et \wedge t\co A;
$$
$$
Proven(t,A) =_{df} \Box Et \wedge t\co A.
$$
However, as for the other two modalities of the basic jstit logic,
their indefinability within JA-STIT is rather obvious and can be
easily shown. On the other hand, $Et$-modality itself does not
seem to be definable within the basic jstit logic. Given all these
facts, the relation between JA-STIT and the basic jstit logic can
be described as follows. The fragment of basic jstit logic given
by the two modalities $\{ Prove(j,t,A), Proven(t,A) \}$ plus the
set of stit and justification modalities can be faithfully
recovered within JA-STIT. This is a maximal fragment of basic
jstit logic that can be recovered within JA-STIT, and JA-STIT
itself is a proper extension of this fragment in terms of
expressive power. In the other direction, $Et$-modality cannot be
recovered within basic jstit logic, which means that only those
fragments of JA-STIT can be recovered within basic jstit logic
which are confined to combinations of modalities borrowed directly
from justification and stit logics.

\subsection{Concluding remarks}

Before we start with the task of axiomatizing JA-STIT, we briefly
mention some facts about its expressive powers which are relevant
to our chosen format of completeness proof. Firstly, it is worth
noting that under the presented semantics some satisfiable
formulas cannot be satisfied over finite models, or even over
infinite models where all histories are finite. The argument for
this is the same as in implicit fragment of basic jstit logic, for
which this claim was proved in \cite{OLimpl} using $K(\Diamond p
\wedge \Diamond\neg p)$ as an example of a formula which is
satisfiable over jstit models but not over jstit models with
finite histories. This already rules out some methods of proving
completeness like filtration method.

Secondly, it turns out that, even though JA-STIT is not, strictly
speaking, a temporal logic, it can still tell something about the
structure of histories generated in a given jstit model. Indeed,
let us define that a jstit model $\mathcal{M}$ is \emph{based on
discrete time} iff every chain in $Hist(\mathcal{M})$ is
isomorphic to an initial segment of $\omega$, the set of natural
numbers. Then it can be shown that:
\begin{proposition}\label{proposition}
The set of JA-STIT formulas valid over the class of
(unirelational) jstit models is a \textbf{proper} subset of the
set of JA-STIT formulas valid over the class of (unirelational)
jstit models based on discrete time.
\end{proposition}
\begin{proof}
We clearly have the subset relation. As for the properness part,
consider the formula $K(\neg\Box Ex \vee \Box Ey) \to (\neg Ex
\vee Ey)$ with $x, y \in PVar$. We show that this formula is not
valid over the class of all unirelational jstit models (hence not
valid over the class of all jstit models either). Consider the
following unirelational model $\mathcal{M} = \langle Tree, \unlhd,
Choice, Act, R, \mathcal{E}, V\rangle$ for the community of a
single agent $j$:
\begin{itemize}
\item $Tree = \{ a, b \} \cup \{ r \in \mathbb{R} \mid 0 < r < 1
\}$;

\item $\unlhd = \{ (a,b) \} \cup \{ (a, r) \mid r \in \mathbb{R}
\cap Tree \} \cup \{ (r, r') \mid r, r' \in \mathbb{R} \cap Tree,
r \leq r' \}$;

\item $Choice^m_j = H_m$ for all $m \in Tree$;

\item $R = \unlhd$;

\item $\mathcal{E}(m, t) = Form$, for all $m \in Tree$ and $t \in
Pol$.

\item $V(p) = \emptyset$ for all $p \in Var$.
\end{itemize}
It is straightforward to see that the above-defined components of
$\mathcal{M}$ satisfy all the constraints imposed on normal jstit
models except possibly those involving $Act$. Before we go on and
define $Act$, let us pause a bit and reflect on the structure of
histories in the model $\mathcal{M}$ that is being defined. We
only have two histories in it, one is $h_1 = (a,b)$ and the other
is $h_2 = \{ a \} \cup \{ r \in \mathbb{R} \mid 0 < r < 1 \}$. So
we define:
\begin{itemize}
\item $Act(a, h_2) = \{ x \}$;

\item $Act(a, h_1) = Act(b, h_1) = \emptyset$;

\item $Act(r, h_2) = \{ x, y \}$ for all $r \in \mathbb{R} \cap
Tree$.
\end{itemize}
Again, most of the constraints on jstit models are now easily seen
to be satisfied.\footnote{Note that this model also satisfies any
possible \emph{constant specification} (as defined in Section
\ref{main}) so that introducing any such specification cannot
affect the counterexample at hand.} The no new proofs guaranteed
constraint is perhaps less straightforward, so we consider it in
some detail. We have, on the one hand, $Act(a, h_1) \cap Act(a,
h_2) = Act(b, h_1) = \emptyset$, so neither $a$, nor $b$ can
falsify the constraint. The only remaining option is that $m \in
\{ r \in \mathbb{R} \mid 0 < r < 1 \}$, say $m = r$. But then the
only history passing through $r$ is $h_2$ and we have, on the
other hand, $\frac{r}{2} \in Tree$, $\frac{r}{2} < r$, and
$Act(\frac{r}{2}, h_2) = Act(r,h_2)$ so that the no new proofs
guaranteed constraint is again verified.

Now, consider $a \in Tree$. The set of $a$'s epistemic
alternatives is $Tree$ itself. We have $\mathcal{M}, a, h_1
\not\models Ex$, therefore $\mathcal{M}, a, h_2 \models \neg\Box
Ex$, whence $\mathcal{M}, a, h_2 \models \neg\Box Ex \vee \Box
Ey$. We also have, of course, that $\mathcal{M}, a, h_1 \models
\neg\Box Ex$ and $\mathcal{M}, a, h_1 \models \neg\Box Ex \vee
\Box Ey$. In the same way, we see that $\mathcal{M}, b, h_1
\models \neg\Box Ex$ and $\mathcal{M}, b, h_1 \models \neg\Box Ex
\vee \Box Ey$

Moreover, if $r$ is a real number strictly between $0$ and $1$,
then $\mathcal{M}, r, h_2 \models Ex$, and, since $h_2$ is the
only history passing through $r$, we get also $\mathcal{M}, r, h_2
\models \Box Ex$, and, further, $\mathcal{M}, r, h_2 \models
\neg\Box Ex \vee \Box Ey$. Thus the formula $\neg\Box Ex \vee \Box
Ey$ holds at every epistemic alternative of $a$ for every history
passing through this alternative. This means that $\mathcal{M}, a,
h_2 \models K(\neg\Box Ex \vee \Box Ey)$. Besides, we have that
$\mathcal{M}, a, h_2 \models Ex \wedge \neg Ey$, so that the pair
$(a, h_2)$ falsifies the formula $K(\neg\Box Ex \vee \Box Ey) \to
(\neg Ex \vee Ey)$ in $\mathcal{M}$.

On the other hand, $K(\neg\Box Ex \vee \Box Ey) \to (\neg Ex \vee
Ey)$ is valid in the class of jstit models based on discrete time
(hence also over unirelational jstit models based on discrete
time). In order to show this, we will assume its invalidity and
obtain a contradiction. Indeed, let $\mathcal{M} = \langle Tree,
\unlhd, Choice, Act, R,R_e, \mathcal{E}, V\rangle$ be a jstit
model based on discrete time such that
$$
\mathcal{M}, m, h \not\models K(\neg\Box Ex \vee \Box Ey) \to
(\neg Ex \vee Ey).
$$
Then we will have both
\begin{equation}\label{E:discr1}
\mathcal{M}, m, h \models K(\neg\Box Ex \vee \Box Ey),
\end{equation}
and
\begin{equation}\label{E:discr2}
\mathcal{M}, m, h \models Ex \wedge \neg Ey.
\end{equation}
By \eqref{E:discr1}, we know that:
\begin{equation}\label{E:discr3}
\mathcal{M}, m, h \models \neg\Box Ex \vee \Box Ey,
\end{equation}
and, by \eqref{E:discr2}, it follows that:
\begin{equation}\label{E:discr4}
\mathcal{M}, m, h \models \neg\Box Ey.
\end{equation}
Therefore, we know by \eqref{E:discr3} that $\mathcal{M}, m, h
\models \neg\Box Ex$, so that there is an $h' \in H_m$ such that
$\mathcal{M}, m, h' \models \neg Ex$. In view of \eqref{E:discr2},
we must have $h \neq h'$, so $H_m$ cannot be a singleton. Since
histories are defined as maximal chains of moments, we know that
$H_{m'}$ is always a singleton when $m' \in Tree$ is
$\unlhd$-maximal. Therefore $m$ cannot be $\unlhd$-maximal and
thus $m$ cannot be the $\unlhd$-last moment along $h$. Since
$\mathcal{M}$ is based on discrete time, consider embedding $f$ of
$h$ into an initial segment of $\omega$. Suppose that $f(m) = n$.
Since $m$ is not the $\unlhd$-last moment along $h$, there must be
an $m' \in h$ such that $f(m') = n + 1$. Since $f$ is an
embedding, this means that $m \lhd m'$ and for no $m'' \in Tree$
it is true that $m \lhd m'' \lhd m'$. By the future always matters
constraint, we know that $R(m,m')$, therefore, by \eqref{E:discr1}
we must have:
\begin{equation}\label{E:discr5}
\mathcal{M}, m', h \models \neg\Box Ex \vee \Box Ey.
\end{equation}
On the other hand, let $g \in H_{m'}$ be arbitrary. Then, by the
absence of backward branching, $g \in H_m$, and, moreover, $g$ is
undivided from $h$ at $m$. Therefore, by the presenting a new
proof makes histories divide constraint, we must have $Act(m, g) =
Act(m,h)$. By \eqref{E:discr2} we know that $x \in Act(m,h)$,
which means that also  $x \in Act(m,g)$. Since $g \in H_{m'}$ was
chosen arbitrarily, the latter means that $x \in \bigcap_{g \in
H_{m'}}(Act(m,g))$, and, by the expansion of presented proofs
constraint, $x \in \bigcap_{g \in H_{m'}}(Act(m',g))$. This, in
turn, yields that:
\begin{equation}\label{E:discr6}
\mathcal{M}, m', h \models \Box Ex,
\end{equation}
whence, in view of \eqref{E:discr5}, it follows that
\begin{equation}\label{E:discr7}
\mathcal{M}, m', h \models \Box Ey.
\end{equation}
The latter means that $y \in \bigcap_{g \in H_{m'}}(Act(m',g))$,
and by the no new proofs guaranteed constraint, it follows that
for some $g \in H_{m'}$ and some $m'' \in g$ such that $m'' \lhd
m'$, we must have $y \in Act(m'',g)$. Now, if $m'' \lhd m'$ it
follows that $m'' \unlhd m$, since $m'$ was chosen as the
immediate $\lhd$-successor of $m$ along $h$. The latter means, by
the expansion of presented proofs, that $y \in Act(m,g)$. Since
$g$ is undivided from $h$ at $m$, this means, by the presenting a
new proof makes histories divide constraint, that $Act(m, g) =
Act(m,h)$ and, further, that $y \in Act(m,h)$. The latter is in
obvious contradiction with \eqref{E:discr2}.
\end{proof}

Proposition \ref{proposition} shows that if one wants to prove the
completeness theorem for JA-STIT by constructing a canonical
model, the histories in this model \emph{both} have to be allowed
to be infinite \emph{and} have to have a rather involved order
structure. This shows that the canonical model used in the
completeness proof that follows below, is not likely to allow for
any major simplifications.

\section{Axiomatic system and soundness}\label{axioms}

We consider the Hilbert-style axiomatic system $\Sigma$ with the
following set of axiomatic schemes:
\begin{align}
&\textup{A full set of axioms for classical propositional
logic}\label{A0}\tag{\text{A0}}\\
&\textup{$S5$ axioms for $\Box$ and $[j]$ for every $j \in
Agent$}\label{A1}\tag{\text{A1}}\\
&\Box A \to [j]A \textup{ for every }j \in Agent\label{A2}\tag{\text{A2}}\\
&(\Diamond[j_1]A_1 \wedge\ldots \wedge \Diamond[j_n]A_n) \to
\Diamond([j_1]A_1 \wedge\ldots \wedge[j_n]A_n)\label{A3}\tag{\text{A3}}\\
&(s\co(A \to B) \to (t\co A \to (s\times t)\co
B)\label{A4}\tag{\text{A4}}\\
&t\co A \to (!t\co(t\co A) \wedge KA)\label{A5}\tag{\text{A5}}\\
&(s\co A \vee t\co A) \to (s+t)\co A\label{A6}\tag{\text{A6}}\\
&\textup{$S4$ axioms for $K$}\label{A7}\tag{\text{A7}}\\
&KA \to \Box K\Box A\label{A8}\tag{\text{A8}}\\
&\Box Et \to K\Box Et\label{A9}\tag{\text{A9}}
\end{align}
The assumption is that in \eqref{A3} $j_1,\ldots, j_n$ are
pairwise different.

To this set of axiom schemes we add the following rules of
inference:
\begin{align}
&\textup{From }A, A \to B \textup{ infer } B;\label{R1}\tag{\text{R1}}\\
&\textup{From }A\textup{ infer }KA;\label{R2}\tag{\text{R2}}\\
&\textup{If $A$ is an instance of (A0)--(A9) and $c \in PConst$,
then infer $c\co A$;}\label{R3}\tag{\text{R3}}\\
&\textup{From }KA \to (\neg\Box Et_1 \vee\ldots\vee\neg\Box Et_n)\notag\\
&\qquad\qquad\textup{ infer } KA \to (\neg Et_1 \vee\ldots\vee\neg
Et_n).\label{R4}\tag{\text{R4}}
\end{align}
Rule \eqref{R3} is obviously \textbf{not} satisfied over the
general class of jstit models. However, we introduce it as an
inheritance of justification logic with its \emph{constant
specifications}. Rule \eqref{R3} gives just one example of such
constant specification, but it serves as a general case in our
situation, since the form of our completeness proof allows for a
straightforward adaptation to any other variant of constant
specification allowed for in justification logic, including the
empty constant specification which would correspond to omitting
\eqref{R3} altogether. On the other hand, should we take the empty
constant specification as our default example, it would not be
clear how to adapt the proof to accommodate non-empty constant
specification, since completeness proof for the empty
specification allows for quite a bit of shortcuts, which are not
available in the more general case. We postpone a more general
discussion of constant specifications till Section \ref{main},
confining ourselves in the meantime to the particular case given
by \eqref{R3}.

In order to adapt the scope of our completeness result to the
presence of \eqref{R3}, we call a (unirelational) jstit model
$\mathcal{M}$ \emph{normal} iff the following condition is
satisfied:
\begin{align*}
(\forall c \in PConst)(\forall m \in Tree)(\{ A \mid A&\text{ is a
substitution case of}\\
 &\text{a scheme in \eqref{A1}--\eqref{A9}}\}
\subseteq \mathcal{E}(m,c)).
\end{align*}
Our goal is now to obtain a strong completeness theorem for
$\Sigma$ w.r.t. the class of normal models. Establishing soundness
mostly reduces to a routine check that every axiom is valid and
that rules preserve validity. We treat the less obvious cases in
some detail:

\begin{theorem}\label{soundness}
Every instance of \eqref{A0}--\eqref{A9} is valid over the class
of normal jstit models. Every application of rules
\eqref{R1}--\eqref{R4} to formulas which are valid over the class
of normal jstit models yields a formula which is valid over the
class of normal jstit models.
\end{theorem}
\begin{proof}
First, note that if $\mathcal{M} = \langle Tree, \unlhd, Choice,
Act, R, R_e, \mathcal{E}, V\rangle$ is a normal jstit model, then
$\langle Tree, \unlhd, Choice, V\rangle$ is a model of stit logic.
Therefore, axioms \eqref{A0}--\eqref{A3}, which were copy-pasted
from the standard axiomatization of \emph{dstit}
logic\footnote{See, e.g. \cite[Ch. 17]{belnap2001facing}, although
$\Sigma$ uses a simpler format closer to that given in
\cite[Section 2.3]{balbiani}.} must be valid. Second, note that if
$\mathcal{M} = \langle Tree, \unlhd, Choice, Act, R, R_e,
\mathcal{E}, V\rangle$ is a normal jstit model, then $\mathcal{M}
= \langle Tree, R, R_e, \mathcal{E}, V\rangle$ is what is called
in \cite[Section 6]{ArtemovN05} a justification model with the
form of constant specification defined by \eqref{R3}\footnote{The
format for the variable assignment $V$ is slightly different, but
this is of no consequence for the present setting.}. This means
that also all of the \eqref{A4}--\eqref{A7} must be valid, whereas
\eqref{R1}--\eqref{R3} must preserve validity, given that all
these parts of our axiomatic system were borrowed from the
standard axiomatization of justification logic . The validity of
other parts of $\Sigma$ will be motivated below in some detail. In
what follows, $\mathcal{M} = \langle Tree, \unlhd, Choice, Act, R,
R_e,\mathcal{E}, V\rangle$ will always stand for an arbitrary
normal jstit model, and $(m,h)$ for an arbitrary element of
$MH(\mathcal{M})$.

As for \eqref{A8}, assume for \emph{reductio} that $\mathcal{M},
m,h \models KA \wedge \Diamond K \Diamond\neg A$. Then
$\mathcal{M}, m,h \models KA$ and also $\mathcal{M}, m,h' \models
K \Diamond\neg A$ for some $h' \in H_m$. By reflexivity of $R$, it
follows that $\Diamond\neg A$ will be satisfied at $(m,h)$ in
$\mathcal{M}$. The latter means that, for some $h'' \in H_m$, $A$
must fail at $(m,h'')$ and therefore, again by reflexivity of $R$,
$KA$ must fail at $(m,h)$ in $\mathcal{M}$, a contradiction.

We consider next \eqref{A9}. If $\Box Et$ is true at $(m,h)$ in
$\mathcal{M}$, then, by definition,

\noindent$t \in \bigcap_{h \in H_m}Act(m,h)$. Now, if $m' \in
Tree$ is such that $R(m,m')$, then, by epistemic transparency of
presented proofs constraint, we must have $t \in \bigcap_{h' \in
H_{m'}}Act(m',h')$ so that for every $g \in H_{m'}$ we will have
$\mathcal{M},m',g \models \Box Et$. Therefore, we must have
$\mathcal{M},m,h \models K\Box Et$ as well.

It only remains to show that \eqref{R4} preserves validity over
normal jstit models. Assume that $KA \to (\neg\Box Et_1
\vee\ldots\vee\neg\Box Et_n)$ is valid over normal jstit models,
and assume also that we have:
\begin{equation}\label{E:e1}
\mathcal{M}, m, h \models KA \wedge Et_1 \wedge\ldots\wedge
Et_n.
\end{equation}
Whence, by the assumed validity, we know that also:
$$
\mathcal{M}, m, h \models \neg\Box Et_1 \vee\ldots\vee\neg\Box
Et_n,
$$
therefore, we can choose a natural $k$ such that $1 \leq k \leq n$
and $\mathcal{M}, m, h \models \neg\Box Et_k$. The latter, in
turn, means that for some $h' \in H_m$ we have that:
\begin{equation}\label{E:e2}
\mathcal{M}, m, h' \models \neg Et_k.
\end{equation}
Comparison between \eqref{E:e1} and \eqref{E:e2} shows that $h
\neq h'$. Therefore, we know that $H_m$ is not a singleton, which
means that $m$ cannot be a $\unlhd$-maximal moment in $Tree$ and
we can choose an $m' \in Tree$ such that $h \in H_{m'}$ and $m'
\rhd m$. By \eqref{E:e1} we know that $t_1,\ldots, t_n \in Act(m,
h)$ and we know that every $g \in H_{m'}$ is undivided from $h$ at
$m$. Therefore, by the presenting a new proof makes histories
divide constraint, we get that $t_1,\ldots, t_n \in Act(m, g)$ for
all $g \in H_{m'}$, hence, by the expansion of presented proofs
constraint, we get that $t_1,\ldots, t_n \in \bigcap_{g \in
H_{m'}}Act(m',g)$. This means that we have, on the one hand:
\begin{equation}\label{E:e3}
\mathcal{M}, m', h \models  \Box Et_1 \wedge\ldots\wedge \Box
Et_n.
\end{equation}
And, on the other, hand, we know that by the future always matters
constraint, we have $R(m,m')$, which also means that, by
\eqref{E:e1} we get that:
\begin{equation}\label{E:e4}
\mathcal{M}, m', h \models KA.
\end{equation}
Taken together, \eqref{E:e3} and \eqref{E:e4} contradict the validity
of $KA \to (\neg\Box Et_1 \vee,\ldots, \vee\neg\Box Et_n)$.
\end{proof}

We then define a \emph{proof} in $\Sigma$ as a finite sequence of
formulas such that every formula in it is either an axiom or is
obtained from earlier elements of the sequence by one of inference
rules. A proof is a proof of its last formula. If an $A \in Form$
is provable in our system, we will write $\vdash_\Sigma A$.
However, since we will not be considering any axiomatic systems
different from $\Sigma$ until Section \ref{main}, the subscript to
$\vdash$ will be suppressed. Similarly, we will simply speak of
consistency and inconsistency meaning consistency and
inconsistency \emph{relative to} $\Sigma$.

The presence of \eqref{R4} in $\Sigma$ complicates the issue of
finding the right notion of an inference from premises and the
right format for Deduction Theorem. Therefore, we cannot just
define that a set $\Gamma \subseteq Form$ is inconsistent iff
$\bot$ is derivable from $\Gamma$. We have to take a little detour
and say that $\Gamma \subseteq Form$ is \emph{inconsistent} iff
for some $A_1,\ldots,A_n \in \Gamma$ we have $\vdash (A_1
\wedge\ldots \wedge A_n) \to \bot$, and we say that $\Gamma$ is
consistent iff it is not inconsistent. $\Gamma$ is
\emph{maxiconsistent} iff it is consistent and no consistent
subset of $Form$ properly extends $\Gamma$.

Even with this slightly non-standard definition of inconsistency,
we can still do many familiar things, e.g. extend consistent sets
with new formulas and eventually make them maxiconsistent. More
precisely, the following lemma holds:

\begin{lemma}\label{elementaryconsistency}
Let $\Gamma \subseteq Form$ be consistent, and let $A, B \in
Form$. Then:
\begin{enumerate}
\item There exists a $\Delta \subseteq Form$ such that $\Delta$ is
maxiconsistent and $\Gamma \subseteq \Delta$.

\item If $\Gamma$ is maxiconsistent, then exactly one element of
$\{A, \neg A \}$ is in $\Gamma$.

\item If $\Gamma$ is maxiconsistent, then $A \vee B \in \Gamma$
iff $(A \in \Gamma$ or $B \in \Gamma)$.

\item If $\Gamma$ is maxiconsistent and $A, (A \to B) \in \Gamma$,
then $B \in \Gamma$.

\item If $\Gamma$ is maxiconsistent, then $A \wedge B \in \Gamma$
iff $(A \in \Gamma$ and $B \in \Gamma)$.
\end{enumerate}
\end{lemma}
\begin{proof} (Part 1) Just as in the standard case, we enumerate the
elements of $Form$ as $A_1,\ldots, A_n,\ldots$ and form the
sequence of sets $\Gamma_1,\ldots, \Gamma_n,\ldots,$ such that
$\Gamma_1 := \Gamma$ and for every natural $i \geq 1$:
\begin{align*}
    \Gamma_{i + 1} :=
    \left\{%
\begin{array}{ll}
    \Gamma_i, & \hbox{ if $\Gamma_i \cup \{ A_ i \}$ is inconsistent;} \\
    \Gamma_i \cup \{ A_ i \}, & \hbox{ otherwise.} \\
\end{array}%
\right.
\end{align*}
We now define $\Delta := \bigcup_{i \geq 1}\Gamma_i$. Of course,
we have $\Gamma \subseteq \Delta$, and, moreover, $\Delta$ is
maxiconsistent. To see this, note that for every $i \geq 1$ the
set $\Gamma_i$ is consistent by construction. Now, if $\Delta$ is
inconsistent, then there must be a valid implication from a finite
conjunction of formulas in $\Delta$ to $\bot$. These formulas must
be mentioned in our numeration of $Form$ so that the valid
implication in question can presented as $\vdash (A_{i_1}
\wedge\ldots \wedge A_{i_n}) \to \bot$ for appropriate natural
$i_1,\ldots, i_n$. Since all of $A_{i_1}, \ldots, A_{i_n}$ are in
$\Delta$, we must have, by the construction of $\Gamma_1,\ldots,
\Gamma_n,\ldots,$ that $A_{i_1}, \ldots, A_{i_n} \in
\Gamma_{max(i_1,\ldots, i_n)}$. But then this latter set must be
inconsistent which contradicts our construction.

Further, if some consistent $\Xi \subseteq Form$ is such that
$\Delta \subset \Xi$, then let $A_n \in \Xi \setminus \Delta$. We
must have then $\Gamma_n \cup \{ A_n \}$ inconsistent, but we also
have $\Gamma_n \cup \{ A_n \} \subseteq \Xi$, which implies
inconsistency of $\Xi$, in contradiction to our assumptions.
Therefore, $\Delta$ is not only consistent, but also
maxiconsistent.

(Part 2) We cannot have both $A$ and $\neg A$ in $\Gamma$, since
we have, of course, $\vdash (A \wedge \neg A) \to \bot$. If, on
the other hand, neither $A$, nor $\neg A$ is in $\Gamma$, then
both $\Gamma \cup \{ A \}$ and $\Gamma \cup \{ \neg A \}$ must be
inconsistent, so that for some $B_1, \ldots, B_n \in \Gamma$ we
will have:
$$
\vdash (B_1\wedge \ldots\wedge B_n \wedge A) \to \bot,
$$
whereas for some $C_1, \ldots, C_k \in \Gamma$ we will have:
$$
\vdash (C_1\wedge \ldots\wedge C_k \wedge \neg A) \to \bot,
$$
whence we get, using \eqref{A0} and \eqref{R1}:
$$
\vdash (C_1\wedge \ldots\wedge C_k) \to A,
$$
and further:
$$
\vdash (B_1\wedge \ldots\wedge B_n \wedge C_1\wedge \ldots\wedge
C_k) \to \bot,
$$
so that $\Gamma$ turns out to be inconsistent, contrary to our
assumptions.

(Part 3) Assume $(A \vee B) \in \Gamma$. If neither $A$ nor $B$
are in $\Gamma$, then, by Part 2, both $\neg A$ and $\neg B$ are
in $\Gamma$. Using \eqref{A0} and \eqref{R1} we get that:
$$
\vdash ((A \vee B) \wedge \neg A \wedge \neg B) \to \bot,
$$
showing that $\Gamma$ is inconsistent, contrary to our
assumptions. In the other direction, if, say $A \in \Gamma$ and
$(A \vee B) \notin \Gamma$, then, by Part 2, we must have $\neg(A
\vee B) \in \Gamma$. Using \eqref{A0} and \eqref{R1} we get that:
$$
\vdash (\neg(A \vee B) \wedge A) \to \bot,
$$
showing, again, that $\Gamma$ is inconsistent, contrary to our
assumptions. The case when $B \in \Gamma$ is similar.

Parts 4 and 5 are similar to Part 3.
\end{proof}

We are now prepared to formulate our main result:

\begin{theorem}\label{completeness}
Let $\Gamma \subseteq Form$. Then $\Gamma$ is consistent iff it is
satisfiable in a normal (unirelational) jstit model.
\end{theorem}

The rest of the paper is mainly concerned with proving Theorem
\ref{completeness}. One part of it we have, of course, right away,
as a consequence of Theorem \ref{soundness}:

\begin{corollary}\label{c-soundness}
If $\Gamma \subseteq Form$ is satisfiable in a normal
(unirelational) jstit model, then $\Gamma$ is consistent.
\end{corollary}
\begin{proof}
Let $\Gamma \subseteq Form$ be satisfiable in a normal jstit model
so that we have, say $\mathcal{M}, m, h \models \Gamma$ for some
$(m,h) \in MH(\mathcal{M})$. If $\Gamma$ were inconsistent this
would mean that for some $A_1,\ldots,A_n \in \Gamma$ we would have
$\vdash (A_1 \wedge\ldots \wedge A_n) \to \bot$. By Theorem
\ref{soundness}, this would mean that:
$$
\mathcal{M}, m, h \models (A_1 \wedge\ldots \wedge A_n) \to \bot,
$$
whence clearly $\mathcal{M}, m, h \models \bot$, which is
impossible. Therefore, $\Gamma$ must be consistent.

Further, if $\Gamma \subseteq Form$ is satisfiable in a normal
unirelational jstit model, then $\Gamma$ must be satisfiable in a
normal jstit model. Hence $\Gamma$ must be consistent by the above
reasoning.
\end{proof}

Before we move further, we mention some theorems in the above
axiom system to be used later in the proof of the main result:
\begin{lemma}\label{theorems}
The following holds for every $A \in Form$, $t \in Pol$, $x \in
PVar$, and $j \in Ag$:
\begin{enumerate}
\item $\vdash t\co A \to \Box t\co A$;

\item  $\vdash KA \to \Box KA$.
\end{enumerate}
\end{lemma}
\begin{proof}
(Part 1) We have:
\begin{align*}
    t\co A &\to !t\co t\co A&&\textup{(by \eqref{A5})}\\
    &\to Kt\co A&&\textup{(by \eqref{A5})}\\
    &\to \Box K\Box t\co A&&\textup{(by \eqref{A8})}\\
    &\to K\Box t\co A&&\textup{(by \eqref{A1})}\\
    &\to \Box t\co A&&\textup{(by \eqref{A7})}
\end{align*}
 Our theorem follows then by transitivity of implication.

(Part 2). By S5 properties of $\Box$ and S4 properties of $K$, we
clearly have

\noindent$\vdash \Box K\Box A \to \Box KA$. Part 2 follows then by
\eqref{A8} and transitivity of implication.
\end{proof}

\section{The canonical model}\label{canonicalmodel}

The main aim of the present section is to prove the inverse of
Corollary \ref{c-soundness}. The method used is a variant of the
canonical model technique, but, due to the complexity of the case,
we do not define our model in one full sweep. Rather, we proceed
piecewise, defining elements of the model one by one, and checking
the relevant constraints as soon, as we have got enough parts of
the model in place. The last subsection proves the truth lemma for
the defined model. As we have already indicated, the model to be
built will be a normal unirelational jstit model, so that $R_e$
will be omitted, or, equivalently, assumed to coincide with $R$.

The ultimate building blocks of $\mathcal{M}$ we will call
\emph{elements}. Before going on with the definition of
$\mathcal{M}$, we define what these elements are and explore some
of their properties.

\begin{definition}\label{element}
An \emph{element} is a sequence of the form
$(\Gamma_1,\ldots,\Gamma_n)$ for some $n\in \omega$ with $n \geq
1$ such that:
\begin{itemize}
\item For every $k \leq n$, $\Gamma_k$ is maxiconsistent;

\item For every $k < n$, for all $A \in Form$, if $KA \in
\Gamma_k$, then $KA \in \Gamma_{k + 1}$;

\item For every $k < n$, for all $t \in Pol$, if $Et \in
\Gamma_k$, then $\Box Et \in \Gamma_{k + 1}$.
\end{itemize}
\end{definition}

We prove the following lemma:
\begin{lemma}\label{elementcontinuation}
Whenever $(\Gamma_1,\ldots,\Gamma_n)$ is an element, there exists
a $\Gamma_{n + 1} \subseteq Form$ such that the sequence
$(\Gamma_1,\ldots,\Gamma_{n + 1})$ is also an element.
\end{lemma}
\begin{proof}
Assume $(\Gamma_1,\ldots,\Gamma_n)$ is an element and consider the
following set:
$$
\Delta := \{ KA \mid KA \in \Gamma_n \} \cup \{ \Box Et \mid Et
\in \Gamma_n \}.
$$

We show that $\Delta$ is consistent. Of course, the set $\{ KA
\mid KA \in \Gamma_n \}$ is consistent since it is a subset of
$\Gamma_n$ and the latter is assumed to be consistent. Further, if
$\Delta$ is inconsistent, then, wlog, for some $KB_1,\ldots, KB_r,
Et_1,\ldots, Et_u \in \Gamma_n$ we will have:
$$
\vdash(KB_1\wedge\ldots \wedge KB_r) \to (\neg\Box Et_1 \vee\ldots
\vee \neg\Box Et_u),
$$
whence, by \eqref{A7}:
$$
\vdash K(B_1\wedge\ldots \wedge B_r) \to (\neg\Box Et_1 \vee\ldots
\vee \neg\Box Et_u),
$$
and further, by \eqref{R4}:
$$
\vdash K(B_1\wedge\ldots \wedge B_r) \to (\neg Et_1 \vee\ldots
\vee \neg Et_u).
$$
The latter formula shows that $\Gamma_n$ is inconsistent which
contradicts the assumption that $(\Gamma_1,\ldots,\Gamma_n)$ is an
element.

Therefore, $\Delta$ must be consistent, and, by Lemma
\ref{elementaryconsistency}.1, it is also extendable to a
maxiconsistent $\Gamma_{n + 1}$. By the choice of $\Delta$, this
means that $(\Gamma_1,\ldots,\Gamma_n, \Gamma_{n + 1})$ must be an
element.
\end{proof}
The structure of elements will be important in what follows. If
$\xi = (\Gamma_1,\ldots, \Gamma_n)$ is an element and an element
$\tau$ is of the form $(\Gamma_1,\ldots, \Gamma_{k})$ with $k <
n$, we say that $\tau$ is a \emph{proper} initial segment of
$\xi$. Moreover, if $k = n -1$, then $\tau$ is the \emph{greatest}
proper initial segment of $\xi$. We define $n$ to be  the
\emph{length} of $\xi$. Furthermore, we define that $\Gamma_n$ is
the end element of $\xi$ and write $\Gamma_n = end(\xi)$.

We now define the canonical model using elements as our building
blocks. We start by defining the following relation $\equiv$:

\begin{align*}
(\Gamma_1,\ldots, \Gamma_n, \Gamma_{n + 1}) \equiv
(\Delta_1,\ldots,& \Delta_n, \Delta_{n + 1}) \Leftrightarrow
(\Gamma_1 = \Delta_1 \& \ldots \& \Gamma_n = \Delta_n
\&\\
 &\& (\forall A \in Form)(\Box A \in \Gamma_{n + 1} \Rightarrow A
\in \Delta_{n + 1}).
\end{align*}

It is routine to check that $\equiv$ is an equivalence relation
given that $\Box$ is an S5 modality. The notation
$[(\Gamma_1,\ldots, \Gamma_n)]_\equiv$ will denote the
$\equiv$-equivalence class generated by $(\Gamma_1,\ldots,
\Gamma_n)$. Since all the elements inside a given
$\equiv$-equvalence class are of the same length, we may extend
the notion of length to these classes setting that the length of
$[(\Gamma_1,\ldots, \Gamma_n)]_\equiv$ also equals $n$.

We now proceed to definitions of components for the canonical
model.

\subsection{$Tree$, $\leq$, and $Hist(\mathcal{M})$}

The first two elements of the canonical model $\mathcal{M}$ are as
follows:

\begin{itemize}
    \item $Tree = \{ \dag \} \cup \{ ([\xi]_\equiv, n)\mid n \in \omega,\,\xi\textup{ is an element}\}$.
    Thus the elements of $Tree$, with the exception of the special
    moment $\dag$, are $\equiv$-equivalence classes of elements
    coupled with natural numbers. Such moments we will call
    \emph{standard} moments, and the left projection of a
    standard moment $m$ we will call its \emph{core} (and write $\overrightarrow{m}$), while the
    right projection of such moment we will call its \emph{height} (and write
    $|m|$). In this way, we get the equality $m = (\overrightarrow{m}, |m|)$
    for every standard $m \in Tree$.
    We further define that the length of a standard moment $m$ is
    the length of its core. For the sake of completeness, we extend the above notions
    to $\dag$ setting both length and height of this moment to $0$ and defining that $\overrightarrow{\dag} = \dag$.
\item We set that $(\forall m \in Tree\setminus\{ \dag \})(\dag
\lhd m \& \neg m \lhd \dag)$. We further set that for any two
standard moments $m$ and $m'$, we have that $m \lhd m'$ iff either
(1) there exists a $\xi \in
    \overrightarrow{m}$ such that for every $\tau\in
    \overrightarrow{m'}$, $\xi$ is a proper initial segment of $\tau$,
     or (2) $ \overrightarrow{m} =  \overrightarrow{m'}$ and $|m'| < |m|$.
    The relation $\unlhd$ is then defined as the reflexive
    companion to $\lhd$.
\end{itemize}

Before we move on to the choice- and justifications-related
components, let us pause to check that the restraints imposed by
our semantics on $Tree$ and $\unlhd$ are satisfied:
\begin{lemma}\label{leq}
The relation $\unlhd$, as defined above, is a partial order on
$Tree$, which satisfies both historical connection and no backward
branching constraints.
\end{lemma}
\begin{proof}
Reflexivity of $\unlhd$ holds by definition. For
\textbf{transitivity}, suppose that $m,m'$, and $m''$ are in
$Tree$ and that we have $m \unlhd m'$ and $m' \unlhd m''$. Then,
if any two moments among $m,m'$ and $m''$ coincide, or if one of
those moments is $\dag$, we must clearly have $m \unlhd m''$. So
suppose that all of  $m,m'$ and $m''$ are standard and pairwise
different so that we have $m \lhd m' \lhd m''$. We have then four
cases to consider:

\emph{Case 1.}  There are $\xi \in \overrightarrow{m}$ and $\tau
\in \overrightarrow{m'}$ such that $\xi$ is a proper initial
segment of every element in $\overrightarrow{m'}$ (and this
clearly includes $\tau$), and $\tau$ is a proper initial segment
of every element in $\overrightarrow{m''}$. It is immediate then
that $\xi$ is a proper initial segment of every element in
$\overrightarrow{m''}$, and $m \lhd m''$ follows.

\emph{Case 2.} We have $|m| > |m'| > |m''|$ and also
$\overrightarrow{m} = \overrightarrow{m'} = \overrightarrow{m''}$.
Then both $|m| > |m''|$ and $\overrightarrow{m} =
\overrightarrow{m''}$ clearly follow so that we get $m \lhd m''$.

\emph{Case 3.} There is a $\xi \in \overrightarrow{m}$ such that
$\xi$ is a proper initial segment of every element in
$\overrightarrow{m'}$. Additionally, we have both
$\overrightarrow{m'} = \overrightarrow{m''}$ and $|m'| > |m''|$.
Then clearly $\xi$ must be a proper initial segment also of every
element in $\overrightarrow{m''}$ so that $m \lhd m''$ holds.

\emph{Case 4.}  There is a $\tau \in  \overrightarrow{m'}$ such
that $\tau$ is a proper initial segment of every element in
$\overrightarrow{m''}$. On the other hand, we have both
$\overrightarrow{m} = \overrightarrow{m'}$ and $|m| > |m'|$. Then,
of course, $\tau$ is also in $\overrightarrow{m}$ and again we get
$m \lhd m''$.

As for \textbf{anti-symmetry}, assume that we have both $m \lhd
m'$ and $m' \lhd m$. Then both $m$ and $m'$ must be standard.
Again we have to consider four cases, and we obtain a
contradiction in each of them, showing that this situation never
arises:

\emph{Case 1.}  There are $\xi\in \overrightarrow{m}$ and $(\tau)
\in \overrightarrow{m'}$ such that $\xi$ is a proper initial
segment of every element in $\overrightarrow{m'}$ and $\tau$ is a
proper initial segment of every element in $\overrightarrow{m}$.
It is clear then that both $\xi$ is a proper initial segment of
$\tau$ and $\tau$ a proper initial segment of $\xi$, which gives
us the contradiction.

\emph{Case 2.} We have  $\overrightarrow{m} = \overrightarrow{m'}$
and also both $|m| > |m'|$ and $|m| < |m'|$. The contradiction is
immediate.

\emph{Case 3.} There is a $\xi \in \overrightarrow{m}$ such that
$\xi$ is a proper initial segment of every element in
$\overrightarrow{m'}$. Besides, we have both $\overrightarrow{m} =
\overrightarrow{m'}$ and $|m'| > |m|$. But then $\xi \in
\overrightarrow{m'}$ and therefore must be its own proper initial
segment, a contradiction.

\emph{Case 4.} There is a $\tau \in \overrightarrow{m'}$ such that
$\tau$ is a proper initial segment of every element in
$\overrightarrow{m}$, and also we have both $\overrightarrow{m} =
\overrightarrow{m'}$ and $|m| > |m'|$. This case is similar to
Case 3.

\textbf{Historical connection} is satisfied since $\dag$ is the
$\unlhd$-least element of $Tree$.

Let us prove the absence of \textbf{backward branching}. Assume
that we have both $m \unlhd m''$ and $m' \unlhd m''$ but neither
$m \unlhd m'$ nor $m' \unlhd m$ holds. This means that all the
three moments are pairwise different and none of them is $\dag$,
otherwise our assumptions about them would be immediately
falsified. Therefore, all the three moments are standard and we
also have $m \neq m'$, $m \lhd m''$, and $m' \lhd m''$. We will
use the familiar fourfold partition of cases:

\emph{Case 1.}  There are $\xi \in \overrightarrow{m}$ and $\tau
\in \overrightarrow{m'}$ such that both $\xi$ and $\tau$ are
proper initial segments of every element in
$\overrightarrow{m''}$. If $\xi = \tau$, then we must have
$\overrightarrow{m} = \overrightarrow{m'}$ since moment cores are
classes of equivalence. Hence we will have $|m| \neq |m'|$, since
$m \neq m'$. But then, depending on whether we have $|m| < |m'|$
or $|m'| < |m|$, we get either $m' \lhd m$ or $m \lhd m'$. On the
other hand, if $\xi$ is different from $\tau$, then either $\xi$
must be a proper initial segment of $\tau$ or vice versa. Assume,
wlog, that $\xi$ is a proper segment of $\tau$. Then $\xi$ is
included in the greatest proper initial segment of $\tau$ and
since every element in $\overrightarrow{m'}$ has the same greatest
proper initial segment, this means that $\xi$ is a proper initial
segment of every element in $\overrightarrow{m'}$ so that $m \lhd
m'$.

\emph{Case 2.}  We have, on the one hand, $\overrightarrow{m} =
\overrightarrow{m''}$ and $|m''| < |m|$, and, on the other hand
$\overrightarrow{m'} = \overrightarrow{m''}$ and $|m''| < |m'|$.
Then we immediately get that $\overrightarrow{m} =
\overrightarrow{m'}$. Further, by $m \neq m'$ we know that either
$|m| < |m'|$ or $|m'| < |m|$ whence we get, respectively, either
$m' \lhd m$ or $m \lhd m'$.

\emph{Case 3.} There is a $\xi \in \overrightarrow{m}$ such that
$\xi$ is a proper initial segment of every element in
$\overrightarrow{m''}$, and, on the other hand, we have both
$\overrightarrow{m'} = \overrightarrow{m''}$ and $|m''| < |m'|$.
Then, of course $\xi$ is also a proper initial segment of every
element in $\overrightarrow{m'}$, and $m \lhd m'$ follows.

\emph{Case 4.} There is a $\tau \in \overrightarrow{m'}$ such that
$\tau$ is a proper initial segment of every element in
$\overrightarrow{m''}$, and, on the other hand, we have both
$\overrightarrow{m} = \overrightarrow{m''}$ and $|m''| < |m|$.
This case is similar to Case 3.
\end{proof}

Before we move on to the other components of the canonical model
$\mathcal{M}$ to be defined in this section, we look into the
structure of $Hist(\mathcal{M})$ as induced by the above-defined
$Tree$ and $\unlhd$. We start by defining a basic sequence of
elements. A \emph{basic sequence} of elements is a set of elements
of the form $\{ \xi_1,\ldots,\xi_n,\ldots, \}$ such that for every
$n \geq 1$:
\begin{itemize}
    \item $\xi_n$ is of length $n$;
    \item $\xi_n$ is the greatest proper initial segment of
    $\xi_{n + 1}$.
\end{itemize}
Basic sequences will be denoted by capital Latin letters $S$, $T$,
and $U$ with subscripts and superscripts when needed. Every given
basic sequence $S$ induces the following $[S] \subseteq Tree$:
$$
[S] = \{\dag\} \cup \bigcup_{n \in \omega}\{ ([\xi_n]_\equiv, k)
\mid k \in \omega \}.
$$
It is immediate that every basic sequence $S$ induces a unique
$[S] \subseteq Tree$ in this way. It is, perhaps, less immediate
that the mapping $S \mapsto [S]$ is injective:
\begin{lemma}\label{injective}
Let $S$, $T$ be basic sequences of elements. Then:
$$
[S] = [T] \Rightarrow S = T.
$$
\end{lemma}
\begin{proof}
Assume that $S = \{ \xi_1,\ldots,\xi_n,\ldots, \}$ and that $T =
\{ \tau_1,\ldots,\tau_n,\ldots, \}$. We will show that $\xi_n =
\tau_n$ for arbitrary $n \in \omega$. Indeed, note that it is
immediate from the definition of $S \mapsto [S]$, that both $[S]$
and $[T]$ contain exactly one moment of length $n + 1$ and height
$0$, and these moments are $([\xi_{n + 1}]_\equiv, 0)$ and
$([\tau_{n + 1}]_\equiv, 0)$, respectively. Therefore, if $[S] =
[T]$, then we must have $([\xi_{n + 1}]_\equiv, 0) = ([\tau_{n +
1}]_\equiv, 0)$, whence, further, $[\xi_{n + 1}]_\equiv = [\tau_{n
+ 1}]_\equiv$ and $\xi_{n + 1}\equiv \tau_{n + 1}$. Therefore,
$\xi_{n + 1}$ and $\tau_{n + 1}$ must share the greatest proper
initial segment which is $\xi_n$ for $\xi_{n + 1}$ and $\tau_n $
for $\tau_{n + 1}$. Since this segment is the same for $\xi_{n +
1}$ and $\tau_{n + 1}$, it follows that $\xi_n = \tau_n$.
\end{proof}

We now move on to a characterization of $Hist(\mathcal{M})$, first
proving a number of technical lemmas:
\begin{lemma}\label{hist-length1}
If $h \in Hist(\mathcal{M})$ and $k \in \omega$, then $h$ contains
at least one moment of length exceeding $k$.
\end{lemma}
\begin{proof}
Suppose otherwise, and let $k \in \omega$ be such that every
moment in $h$ has length at most $k$. We may assume that this is
the least such $k$ so that some elements of the length $k$ are
actually in $h$. We have to consider two cases then:

\emph{Case 1}. $k = 0$. Then $h =\{ \dag \}$. Take any
maxiconsistent $\Gamma \subseteq Form$, it is immediate that
$(\Gamma)$ is an element. Then $([(\Gamma)]_\equiv, 0) \in Tree$
and, moreover $\dag \lhd ([(\Gamma)]_\equiv, 0)$, so that $\{
\dag,([(\Gamma)]_\equiv, 0)\}$ is a $\unlhd$-chain properly
extending $h$, which contradicts the maximality of $h$.

\emph{Case 2}. $k > 0$. Then take an arbitrary moment $m$ of the
length $k$ in $h$, say $m =
([(\Gamma_1,\ldots,\Gamma_k)]_\equiv,n)$. Then
$([(\Gamma_1,\ldots,\Gamma_k)]_\equiv,0)$ is an $\unlhd$-upper
bound for $h$. Indeed, we clearly have $m \unlhd
([(\Gamma_1,\ldots,\Gamma_k)]_\equiv,0)$. Now, if $m' \in h$, then
either $m' \unlhd m$, or $m \lhd m'$. If $m' \unlhd m$, then, by
transitivity, $m' \unlhd ([(\Gamma_1,\ldots,\Gamma_k)]_\equiv,0)$
and we are done. If $m \lhd m'$, then we cannot have any $\xi \in
\overrightarrow{m}$ such that $\xi$ is a proper initial segment of
every element in $\overrightarrow{m'}$ since every such $\xi$ is
of length $k$ and this would mean that elements in
$\overrightarrow{m'}$ must have a length greater than $k$, which
contradicts the choice of $m'$. Therefore, we must have
$\overrightarrow{m} =  \overrightarrow{m'}$ and also $|m'| < |m|$.
But then also  $m' \unlhd ([(\Gamma_1,\ldots,\Gamma_k)]_\equiv,0)$
clearly follows.

Now, using Lemma \ref{elementcontinuation}, we can choose a
$\Gamma_{k + 1} \subseteq Form$ such that
$(\Gamma_1,\ldots,\Gamma_k,\Gamma_{k + 1})$ is an element.
Consider then $m'' = ([(\Gamma_1,\ldots,\Gamma_k,\Gamma_{k +
1})]_\equiv,0) \in Tree$. We obviously have $m'' \notin h$ since
the length of $m''$ is $k + 1$. On the other hand, we have, by
definition of $\unlhd$, that
$([(\Gamma_1,\ldots,\Gamma_k)]_\equiv,0) \lhd m''$. Hence $h \cup
\{ m'' \}$ is a $\unlhd$-chain properly extending $h$, which,
again, contradicts the maximality of $h$.
\end{proof}
\begin{lemma}\label{hist-length2}
If $h \in Hist(\mathcal{M})$ and $k \in \omega$, then $h$ contains
at least one moment of the length $k$.
\end{lemma}
\begin{proof}
Take an arbitrary $k\in \omega$. If $k = 0$, then the lemma holds,
since $\dag$, being the $\unlhd$-least moment in $Tree$, is of
course in $h$. Assume that $k > 0$. There are two cases to
consider then.

\emph{Case 1}. For every $n + 1 \in \omega$ it is true that
whenever there is a moment of the length $n + 1$ in $h$, then
there is also a moment of length $n$ in $h$. Then our lemma
follows from Lemma \ref{hist-length1}.

\emph{Case 2}. There is an $n + 1 \in \omega$ such that some $m
\in Tree$ of the length $n + 1$ is in $h$, but there are no
moments of the length $n$ in $h$. Then consider $m$, say $m =
([(\Gamma_1,\ldots,\Gamma_n, \Gamma_{n + 1})]_\equiv, r)$. We show
then that $m' = ([(\Gamma_1,\ldots,\Gamma_n)]_\equiv, 0)$ must be
in $h$ as well, since  $h \cup \{ m' \}$ is a $\unlhd$-chain and
$h$ is maximal. Indeed, we have $m' \lhd m$, since
$(\Gamma_1,\ldots,\Gamma_n)$ is a proper initial segment of every
element in $\overrightarrow{m}$. Now, if $m'' \in h$, then either
$m \unlhd m''$, or $m'' \lhd m$. If $m \unlhd m''$, then of course
$m' \lhd m''$ by transitivity.  If, on the other hand, $m'' \lhd
m$, then, by the absence of backward branching, either $m'' \unlhd
m'$ or $m' \unlhd m$.

Thus we have shown that $m' \in h$, and since the length of $m'$
equals $n$, this gives us a contradiction with the hypothesis of
Case 2.
\end{proof}
\begin{lemma}\label{hist-length3}
Assume that $h \in Hist(\mathcal{M})$, that $k \in \omega$, and
that $m, m' \in h$ are of the length $k$. Then
$\overrightarrow{m}=\overrightarrow{m'}$.
\end{lemma}
\begin{proof}
We may assume that $k > 0$ since there is only one core of length
$0$. If $m, m' \in h$ are standard moments, then either $m \unlhd
m'$ or $m' \unlhd m$. Assume, wlog, that $m \unlhd m'$. Then there
is no $\xi \in \overrightarrow{m}$ such that $\xi$ is a proper
initial segment of every element in $\overrightarrow{m'}$, since
the length of $\xi$ is equal to the length of elements in
$\overrightarrow{m'}$. Therefore, we must have
$\overrightarrow{m}=\overrightarrow{m'}$ by definition of
$\unlhd$.
\end{proof}
We now offer the following characterization of
$Hist(\mathcal{M})$:
\begin{lemma}\label{structure}
The following statements hold:

\begin{enumerate}
\item If $S = \{ \xi_1,\ldots,\xi_n,\ldots, \}$ is a basic
sequence, then $[S] \in Hist(\mathcal{M})$, and the following
presentation gives $[S]$ in the $\unlhd$-ascending order:
$$
\dag,\ldots,([\xi_1]_\equiv, k),\ldots, ([\xi_1]_\equiv,
0),\ldots,([\xi_n]_\equiv, k),\ldots, ([\xi_n]_\equiv, 0),\ldots,
$$

\item $Hist(\mathcal{M}) = \{ [S] \mid S\textup{ is a basic
sequence}\}$.
\end{enumerate}
\end{lemma}
\begin{proof} (Part 1). It is quite easy to see that for a given
basic sequence $S = \{ \xi_1,\ldots,\xi_n,\ldots, \}$, $[S]$ is a
$\unlhd$-chain and that Part 1 of the Lemma represents this chain
in the ascending order. We focus on maximality of $[S]$ as a
$\unlhd$-chain. Suppose $m \in Tree$ is such that $m \notin [S]$,
but $[S] \cup \{ m \}$ is still a $\unlhd$-chain. Then $m$ must be
standard, since $\dag$ is already in $[S]$. Suppose $m =
([\tau]_\equiv, k)$ for some element $\tau$ and $k \in \omega$,
and suppose that the length of $m$ is $n \geq 1$. Consider then
$([\xi_{n+1}]_\equiv, 0) \in [S]$. Since $[S] \cup \{ m \}$ is a
$\unlhd$-chain we must have either $([\xi_{n+1}]_\equiv, 0) \unlhd
m$ or $m \lhd [\xi_{n+1}]_\equiv, 0)$. But the length of
$([\xi_{n+1}]_\equiv, 0)$ is greater than the length of $m$,
therefore $[\xi_{n+1}]_\equiv \neq \overrightarrow{m}$ and also no
element in $[\xi_{n+1}]_\equiv$ can be a proper initial segment of
any element in $\overrightarrow{m}$. Therefore, we cannot have
$([\xi_{n+1}]_\equiv, 0) \unlhd m$ and must then get $m \lhd
([\xi_{n+1}]_\equiv, 0)$. Given that we have shown
$[\xi_{n+1}]_\equiv \neq \overrightarrow{m}$, $m \lhd
([\xi_{n+1}]_\equiv, 0)$ must mean that some element $\tau \in
\overrightarrow{m}$ is a proper initial segment of every element
in $([\xi_{n+1}]_\equiv$ including $\xi_{n+1}$. Since the length
of $\xi_{n+1}$ is $n + 1$ and the length of $m$ is n, this means
that $\tau'$ must be the greatest proper initial segment of
$\xi_{n+1}$. But the greatest proper initial segment of
$\xi_{n+1}$ is $\xi_n$, therefore $\tau' = \xi_n$ and,
consequently, $m = ([\tau]_\equiv, k) = ([\xi_n]_\equiv, k) \in
[S]$, which contradicts the choice of $m$.

(Part 2). It follows from Part 1 that  $Hist(\mathcal{M})
\supseteq \{ [S] \mid S\textup{ is a basic sequence}\}$, so we
only need to show the inverse inclusion. So, choose an arbitrary
$h \in Hist(\mathcal{M})$. Consider the set
$$
core(h) = \{ \overrightarrow{m} \mid m \in h\}.
$$
It follows from Lemmas \ref{hist-length2} and \ref{hist-length3}
that $core(h)$ contains exactly one moment core of the length $n$
for every $n \in \omega$. Therefore, $core(h)$ has the form $\{
\dag, \alpha_1,\ldots,\alpha_n,\ldots,  \}$, where every
$\alpha_k$ is an equivalence class of elements of length $k$. We
now claim that if $k \geq 2$, then there is a $\xi_{k - 1} \in
\alpha_{k - 1}$ such that $\xi_{k - 1}$ is a proper initial
segment of every element in $\alpha_k$. Indeed, we know that for
some $r, r' \in \omega$ the moments $(\alpha_{k - 1}, r),
(\alpha_k, r')$ are in $h$. We cannot have $(\alpha_k, r') \unlhd
(\alpha_{k - 1}, r)$ since the length of $\alpha_{k - 1}$ is
strictly less than the length of $\alpha_k$. Therefore, since $h$
is a chain, we must have $(\alpha_{k - 1}, r) \lhd (\alpha_k,
r')$, and, again by length considerations, there must be a $\xi_{k
- 1} \in \alpha_{k - 1}$ such that $\xi_{k - 1}$ is a proper
initial segment of every element in $\alpha_k$.

So we choose such a $\xi_{k - 1} \in \alpha_{k - 1}$ for every $k
\geq 2$. In this way we obtain the sequence $S = \{
\xi_1,\ldots,\xi_n,\ldots, \}$ with the following properties:
\begin{enumerate}
\item For all $k \geq 1$, $\xi_k \in \alpha_k$ (so that $\alpha_k
= [\xi_k]_\equiv$ and $\xi_k$ itself is therefore of the length
$k$);

\item For all $k \geq 1$, $\xi_k$ is a proper initial segment of
every element in $\alpha_{k + 1}$.
\end{enumerate}
Now, for given $k \geq 1$, since $\xi_k$ is a proper initial
segment of every element in $\alpha_{k + 1}$, then $\xi_k$ is also
a proper initial segment of $\xi_{k + 1}$. And since the lengths
of $\xi_k$ and $\xi_{k + 1}$ are $k$ and $k + 1$, respectively,
then $\xi_k$ is the greatest proper initial segment of $\xi_{k +
1}$. This means that the sequence $S = \{
\xi_1,\ldots,\xi_n,\ldots, \}$ is in fact a basic sequence. We now
show that $[S] \subseteq h$ and since, by Part 1, $[S]$ is itself
a history, this will mean that $[S] = h$, and that, given that $h$
was chosen arbitrarily, we will be done.

Indeed, assume that $m \in [S]$. If $m = \dag$, then of course $m
\in h$ by maximality of $h$, since $\dag$ is the $\unlhd$-least
element in $Tree$. Therefore, assume that $m$ is standard, say $m
= ([\xi_n]_\equiv, k)$. Take an arbitrary $m' \in h$. We will show
that we either have $m \unlhd m'$ or $m' \unlhd m$. In the case
when $m' = \dag$ we trivially get $m' \lhd m$ so we assume that
$m'$ is standard so that for some appropriate $k',n' \in \omega$
we must have $m' = ([\xi_{n'}]_\equiv, k')$. We have then three
cases to consider:

\emph{Case 1}. $n' < n$. Then $\xi_{n'}$ must be a proper initial
segment of every element in $[\xi_n]_\equiv$, and we immediately
get $m' \lhd m$.

\emph{Case 2}. $n < n'$. This case is an inversion of Case 1,
giving us that $m \lhd m'$.

\emph{Case 3}. $n = n'$. Then $\overrightarrow{m} =
\overrightarrow{m'}$ and, depending on whether we have $k < k'$,
$k' < k$, or $k = k'$ we obtain that $m' \lhd m$, $m \lhd m'$, or
$m = m'$, respectively.

Thus we have shown that $h \cup \{ m\}$ is an $\unlhd$-chain,
whence, by the maximality of $h$, it follows that $m \in h$. And
since $m \in [S]$ was chosen arbitrarily, this means that $[S]
\subseteq h$ and therefore $[S] = h$, as desired.
\end{proof}
It follows from Lemmas \ref{structure} and \ref{injective} that
not only every basic sequence generates a unique $h \in
Hist(\mathcal{M})$, but also for every $h \in Hist(\mathcal{M})$
there exists a unique basic sequence $S$ such that $h = [S]$. We
will denote this unique $S$ for a given $h$ by $]h[$. It is
immediate from Lemmas \ref{structure} and \ref{injective} that for
every $h \in Hist(\mathcal{M})$, $h = [(]h[)]$. Likewise, for
every basic sequence $S$, we have $S = ]([S])[$. As a further
useful piece of notation, we introduce the notion of
\emph{intersection} of a standard moment $m$ with a history $h \in
H_m$. Assume that $m$ is of the length $n$ and that $]h[  = \{
\xi_1,\ldots,\xi_n,\ldots, \}$. Then $m$ must be of the form
$([\xi_n]_\equiv, k)$ for some $k\in\omega$, and we will also have
$\overrightarrow{m} \cap ]h[ = \{ \xi_n \}$. We now define the
only member of the latter singleton as the result $m \sqcap h$ of
the intersection of $m$ and $h$, setting $m \sqcap h = \xi_n$. It
can be shown that for any element $\xi$ in the core of a given
standard moment $m$ there exists an $h \in H_m$ such that $\xi = m
\sqcap h$:
\begin{lemma}\label{sqcap}
Let $(\Gamma_1,\ldots,\Gamma_k)$ be an element. Then, for every $n
\in \omega$ there is at least one history $h \in
H_{([(\Gamma_1,\ldots,\Gamma_k)]_\equiv, n)}$ such that
$([(\Gamma_1,\ldots,\Gamma_k)]_\equiv, n) \sqcap h =
(\Gamma_1,\ldots,\Gamma_k)$.
\end{lemma}
\begin{proof}
Using Lemma \ref{elementcontinuation} and axiom of choice, we
successively choose $\Gamma_{k + 1},\ldots, \Gamma_{k + l},
\ldots, \subseteq Form$ such that all of the structures
$$
(\Gamma_1,\ldots,\Gamma_k,\Gamma_{k + 1}),\ldots,
(\Gamma_1,\ldots,\Gamma_k,\Gamma_{k + 1},\ldots, \Gamma_{k + l}),
\ldots,
$$
are elements. But then, it is obvious that the set:
\begin{align*}
   S =\{(\Gamma_1),\ldots, (\Gamma_1,\ldots,\Gamma_k),(\Gamma_1,\ldots,\Gamma_k,\Gamma_{k +
   1}),\ldots,
(\Gamma_1,\ldots,\Gamma_k,\Gamma_{k + 1},\ldots, \Gamma_{k + l}),
\ldots,\}
\end{align*}
is a basic sequence and $[(\Gamma_1,\ldots,\Gamma_k)]_\equiv, n)
\in [S]$ so that $[S] \in H_{([(\Gamma_1,\ldots,\Gamma_k)]_\equiv,
n)}$. Further, it is clear that
$[(\Gamma_1,\ldots,\Gamma_k)]_\equiv, n) \sqcap [S] =
(\Gamma_1,\ldots,\Gamma_k)$, as desired.
\end{proof}

We offer some general remarks on what we have shown thus far.
Lemma \ref{structure} shows that every history in the canonical
model has a uniform order structure which can be otherwise
described as follows. If $L$ and $L'$ are two linear orders then
let $L \oplus L'$ be a copy of $L$ with a copy of $L'$ appended at
the end, let $L\otimes L'$ be the result of replacement of every
element in $L'$ with a disjoint copy of $L$, and let $L^\ast$ be
the inversion of $L$. Also, for any $n \in \omega$, let
$(0,\ldots,n)$ be the first $n + 1$ natural numbers with their
natural order. Then Lemma \ref{structure} tells us that every
history in the canonical model is ordered in the type of $(0)
\oplus(\omega^\ast \otimes \omega)$. Also, note that it follows
from Lemma \ref{structure} that for every ordered couple of
natural numbers $(k,n)$ with $k > 0$, every given history $h$
contains exactly one moment of length $k$ and height $n$. Another
general observation is that histories in $\mathcal{M}$ can only
branch off at moments of height $0$, so that at moments of other
heights all the histories remain undivided. This last fact does
not follow from the lemmas proved thus far and we end this
subsection with its proof, also establishing a couple of technical
facts to be used later:

\begin{lemma}\label{property1}
Let $m, m' \in Tree$, and let $h \in H_m$. If $\overrightarrow{m}
= \overrightarrow{m'}$, then $h \in H_{m'}$ and also $m \sqcap h =
m' \sqcap h$.
\end{lemma}
\begin{proof}
If $h \in H_m$, then there is an element $\xi \in
\overrightarrow{m} = \overrightarrow{m'}$ such that $\xi \in ]h[$.
For this element we will also have $\xi = m \sqcap h$. Since $\xi
\in \overrightarrow{m'}$, we further get that $m' = ([\xi]_\equiv,
|m'|)$. It follows, by $\xi \in ]h[$, that $m' \in [(]h[)] = h$ so
that $h \in H_{m'}$. Now, consider $\overrightarrow{m'} \cap ]h[$.
We know that this set must be a singleton with $m' \sqcap h$ as
its only element, and we know also that $\{\xi\} =
\overrightarrow{m} \cap ]h[ = \overrightarrow{m'} \cap ]h[$.
Therefore, $m' \sqcap h = \xi = m \sqcap h$ and thus we are done.
\end{proof}
\begin{corollary}\label{propertycorollary}
If $h \in H_m$ and $m = (\overrightarrow{m}, k + 1)$, then for the
$m' = (\overrightarrow{m}, k)$ it is true that $h \in H_{m'}$.
\end{corollary}
\begin{proof} Immediate from Lemma \ref{property1}.
\end{proof}
\begin{corollary}\label{branching}
Let $m \in Tree$ be such that $|m| > 0$, and let $h,h' \in H_m$.
Then $h$ and $h'$ are undivided at $m$.
\end{corollary}
\begin{proof}
Since $|m| > 0$, we know that $m = k + 1$ for some $k \in \omega$.
Then, by Corollary \ref{propertycorollary}, we must have $h,h' \in
H_{m'}$ for $m' = (\overrightarrow{m}, k)$. It remains to notice
that we clearly have $m \lhd m'$.
\end{proof}

\subsection{$Choice$}

We now define the choice structures of our canonical model:
\begin{itemize}
    \item $Choice^m_j(h) = \{ h' \mid h' \in H_m\,
    (\forall A \in Form)([j]A \in end(h \sqcap m) \Rightarrow A \in end(h' \sqcap
    m))\}$, if $m \neq \dag$ and $|m| = 0$;
    \item $Choice^m_j = H_m$, otherwise.
\end{itemize}
Since for every $j \in Ag$, $[j]$ is an S5-modality, $Choice$
induces a partition on $H_m$ for every given $m \in Tree$. We
check that the choice function verifies the relevant semantic
constraints:
\begin{lemma}\label{choice}
The tuple $\langle Tree, \unlhd, Choice\rangle$, as defined above,
verifies both independence of agents and no choice between
undivided histories constraints.
\end{lemma}
\begin{proof}
We first tackle \textbf{no choice between undivided histories}.
Consider a moment $m$ and two histories $h, h' \in H_m$ such that
$h$ and $h'$ are undivided at $m$. Since the agents' choices are
only non-vacuous at moments represented by standard moments of
height $0$, we may safely assume that $m$ is such a moment. Since
$h$ and $h'$ are undivided at $m$, this means that there is a
moment $m'$ such that $m \lhd m'$ and $m'$ is shared by $h$ and
$h'$. Hence we know that also $m'$ is standard. Suppose the length
of $m$ is $n$ and the length of $m'$ is $n'$. Then $n < n'$ since
$m$ is of height $0$ and therefore has no equivalence classes of
elements of length $n$ above itself. Therefore, $h \sqcap m$ is
the initial segment of length $n$ of $h \sqcap m'$, and similarly,
$h' \sqcap m$ is the initial segment of length $n$ of $h' \sqcap
m'$. But both $h \sqcap m'$ and $h' \sqcap m'$ are, by definition,
in $\overrightarrow{m'}$, therefore, they must share the greatest
proper initial segment. Hence, their initial segments of length
$n$ must coincide as well, and we must have $h \sqcap m = h'
\sqcap m$, whence $end(h \sqcap m) = end(h' \sqcap m)$. Now, if $j
\in Ag$ and $[j]A \in end(h \sqcap m)$, then, by \eqref{A1} and
maxiconsistency of $end(h \sqcap m)$, we will have also $A \in
end(h \sqcap m) = end(h' \sqcap m)$, and thus $h' \in
Choice^m_j(h)$, so that $Choice^m_j(h) = Choice^m_j(h')$ since
$Choice$ is a partition of $H_m$.

Consider, next, the \textbf{independence of agents}. Let $m \in
Tree$ and let $f$ be a function on $Ag$ such that $\forall j \in
Ag(f(j) \in Choice^m_j)$. We are going to show that in this case
$\bigcap_{j \in Ag}f(j) \neq \emptyset$. If $m$ is not a standard
moment of height $0$, then this is obvious, since every agent will
have a vacuous choice. We treat the case when $m$ is a standard
moment of height $0$. Assume that $m = ([(\Gamma_1,\ldots,
\Gamma_{n + 1})]_\equiv,0)$. By \eqref{A1} we know that there is a
set $\Delta$ of formulas of the form $\Box A$ which is shared by
all sets of the form $end(\xi)$ with $\xi \in \overrightarrow{m}$
in the sense that if $\xi \in \overrightarrow{m}$, then $\Box A
\in end(\xi)$ iff $\Box A \in \Delta$. By the same axiom scheme
and Lemma \ref{sqcap}, we also know that for every $j \in Ag$
there is set $\Delta_j$ of formulas of the form $[j]A$ which is
shared by all sets of the form $end(\xi)$ such that $\exists h(h
\in f(j) \wedge \xi = m \sqcap h)$. More precisely:
$$
\xi \in \overrightarrow{m} \Rightarrow (\exists h(h \in f(j)
\wedge \xi = m \sqcap h) \Leftrightarrow (\forall A \in Form)([j]A
\in end(\xi) \Leftrightarrow [j]A \in \Delta_j)).
$$

We now consider the set $\Delta \cup \bigcup\{ \Delta_j\mid j \in
Ag \}$ and show its consistency. Indeed, if this set is
inconsistent, then, wlog, we would have a provable formula of the
following form:

\begin{equation}\label{E:c1}
\vdash (\Box A \wedge \bigwedge_{j \in Ag}[j]A_j) \to \bot.
\end{equation}

But then, choose for every $j \in Ag$ an element $\xi_j \in
\overrightarrow{m}$ such that
$$
(\forall A \in Form)([j]A \in end(\xi_j) \Leftrightarrow [j]A \in
\Delta_j).
$$
This is possible, since we may simply choose an arbitrary $h_j \in
f(j)$ and set $\xi_j: = m \sqcap h_j$. Then we will have $[j]A_j
\in \xi_j$ for every $j \in Ag$. Next, consider $\Gamma_{n + 1}$.
Since

\noindent$m = ([(\Gamma_1,\ldots, \Gamma_{n + 1})]_\equiv,0)$ and
$\Box$ is an S5-modality, we must have:
$$
\{ \Diamond[j]A_j \in Ag \} \subseteq \Gamma_{n + 1},
$$
whence, by Lemma \ref{elementaryconsistency}.5:
$$
\bigwedge_{j \in Ag}\Diamond[j]A_j \in \Gamma_{n + 1},
$$
and further, by \eqref{A3} and Lemma
\ref{elementaryconsistency}.4:
$$
\Diamond\bigwedge_{j \in Ag}[j]A_j \in \Gamma_{n + 1}.
$$
Also, by definition of $\Delta$ and the fact that
$(\Gamma_1,\ldots, \Gamma_{n + 1}) \in \overrightarrow{m}$, we get
successively:
$$
\Box A \in  \Gamma_{n + 1},
$$
then, by Lemma \ref{elementaryconsistency}.5:
$$
\Box A \wedge \Diamond\bigwedge_{j \in Ag}[j]A_j \in \Gamma_{n +
1},
$$
and finally, by the fact that $\Box$ is an S5-modality:

\begin{equation}\label{E:c2}
\Diamond(\Box A \wedge \bigwedge_{j \in Ag}[j]A_j) \in \Gamma_{n +
1}.
\end{equation}

From \eqref{E:c1}, together with \eqref{E:c2}, it follows by S5
reasoning for $\Box$ that $\Diamond\bot \in \Gamma_{n + 1}$, so
that, again by S5 properties of $\Box$ and Lemma
\ref{elementaryconsistency}.4, it follows that $\bot \in \Gamma_{n
+ 1}$, which is in contradiction with maxiconsistency of
$\Gamma_{n + 1}$.

Hence $\Delta \cup \bigcup\{ \Delta_j\mid j \in Ag \}$ is
consistent, and we can extend it to a maxiconsistent $\Xi$. We now
consider $(\Gamma_1,\ldots, \Gamma_n, \Xi)$ and show that it is in
fact an element. Indeed, if $KA \in \Gamma_n$, then $KA \in
\Gamma_{n + 1}$ by definition of an element. But then $\Box KA \in
\Gamma_{n + 1}$ by Lemma \ref{theorems}.2 and maxiconsistency of
$\Gamma_{n + 1}$, whence $\Box KA \in \Delta$ and, therefore,
$\Box KA \in \Xi$. By \eqref{A1} and maxiconsistency of $\Xi$ we
get then $KA \in \Xi$. Similarly, if $Et \in \Gamma_n$, then $\Box
Et \in \Gamma_{n + 1}$ by definition of an element. But then $\Box
Et \in \Delta$ and, therefore, $\Box Et \in \Xi$.

Therefore, $(\Gamma_1,\ldots, \Gamma_n, \Xi)$ is an element and
since, moreover, $\Delta \subseteq \Xi$, then also
$(\Gamma_1,\ldots, \Gamma_n, \Xi) \in \overrightarrow{m}$ so that
$m = ([(\Gamma_1,\ldots, \Gamma_n, \Xi)]_\equiv,0)$. Using Lemma
\ref{sqcap}, we can choose a $g \in H_m$ such that $g \sqcap m =
(\Gamma_1,\ldots, \Gamma_n, \Xi)$. We also know that for every $j
\in Ag$, there is a history $h_j \in f(j)$ such that $h_j \sqcap m
= \xi_j$ by the choice of $\xi_j$. Therefore, for every $j \in
Ag$, $Choice^m_j(h_j) = f(j)$. Also, if $[j]A \in end(\xi_j) =
end(h_j \sqcap m)$, then $[j]A \in \Delta_j$, hence $[j]A \in \Xi
= end(g \sqcap m)$, therefore, by \eqref{A1}, $A \in end(g \sqcap
m)$. Thus we get that $g \in \bigcap_{j \in Ag}Choice^m_j(h_j) =
\bigcap_{j \in Ag}f(j)$ so that the independence of agents is
verified.
\end{proof}

\subsection{$R$ and $\mathcal{E}$}

We now define the justifications-related elements of our canonical
model. We first define $R$ as follows:
\begin{itemize}
    \item $R(([(\Gamma_1,\ldots,\Gamma_n, \Gamma)]_\equiv, k), m')\Leftrightarrow$

     $\qquad\qquad\qquad\quad\Leftrightarrow(m' \neq \dag)\&(\forall
    \tau \in \overrightarrow{m'})(\forall A \in Form)(KA \in \Gamma \Rightarrow KA \in
    end(\tau))$;
    \item $R(\dag,m)$, for all $m \in Tree$.
\end{itemize}

Now, for the definition of $\mathcal{E}$:
\begin{itemize}
    \item For
    all $t \in Pol$: $\mathcal{E}(\dag, t) = \{ A \in Form \mid
\vdash t\co A \}$;

\item For all $t \in Pol$ and $m \neq \dag$:
\begin{align*}
    &(\forall A \in Form)(A \in \mathcal{E}(m, t)
    \Leftrightarrow (\forall \xi \in \overrightarrow{m})(t\co A \in end(\xi))).
\end{align*}
\end{itemize}
We start by mentioning a straightforward corollary to the above
definition:
\begin{lemma}\label{proven}
For all $m \in Tree$ and $t \in Pol$ it is true that $\{ A \in
Form \mid \vdash t\co A \} \subseteq \mathcal{E}(m,t)$.
\end{lemma}
\begin{proof}
This holds simply by the definition of $\mathcal{E}$ when $m =
\dag$. If $m \neq \dag$, then, for every $\xi \in
\overrightarrow{m}$, $end(\xi)$ is a maxiconsistent subset of
$Form$ and must contain every provable formula.
\end{proof}

Note that since we know that for every instance $A$ of one of
axiom schemes in the list \eqref{A0}--\eqref{A9}, it is true that
$\vdash c\co A$ for every $c \in PConst$ (by \eqref{R3}), it
follows, among other things, that the above-defined function
$\mathcal{E}$ satisfies the additional normality condition on
jstit models.

\begin{lemma}\label{r}
The relation $R$, as defined above, is a preorder on $Tree$, and,
together with $\unlhd$, verifies the future always matters
constraint.
\end{lemma}
\begin{proof}
It is straightforward to check that $R$, as defined above, is a
preorder on $Tree$, using \eqref{A7} and \eqref{A8}. Let us look
into why future always matters constraint is verified as well.
Assume $m \in Tree$. If $m = \dag$, then it is connected to all
the elements in $Tree$ by both $\unlhd$ and $R$, so this moment
cannot falsify the constraint. Let us assume that $m \neq \dag$,
say $m = ([(\Gamma_1,\ldots, \Gamma_n)]_\equiv,k)$. If $m \unlhd
m'$, then $m'$ must be also standard. Now, if $\overrightarrow{m}
= \overrightarrow{m'}$ and $KA \in \Gamma_n$, then, by
maxiconsistency of $\Gamma_n$ and Lemma \ref{theorems}.2, we must
also have $\Box KA \in \Gamma_n$, which, by definition of
$\equiv$, means that $KA \in end(\xi)$ for every $\xi
\in\overrightarrow{m} = \overrightarrow{m'}$, and thus we get that
$R(m,m')$. The other option is that $(\Gamma_1,\ldots, \Gamma_n)$
is a proper initial segment of every element in $m'$, so that we
may assume, wlog, that $m' = ([(\Gamma_1,\ldots,
\Gamma_{n'})]_\equiv,k')$ for some $n' > n$. But then take an
arbitrary $A \in Form$. If $KA \in \Gamma_n$, then, since
$(\Gamma_1,\ldots, \Gamma_{n'})$ is an element, $KA \in
\Gamma_{n'}$. Moreover,  by maxiconsistency of $\Gamma_{n'}$ and
Lemma \ref{theorems}.2, we will have $\Box KA \in \Gamma_{n'}$.
Now, by definition of $\equiv$, we get $KA \in end(\tau)$ for any
given $\tau \in \overrightarrow{m'}$. It follows that, again, we
have $R(m, m')$ as desired.
\end{proof}

We further check that the semantical constraints for $\mathcal{E}$
are verified:
\begin{lemma}\label{e}
The function $\mathcal{E}$, as defined above, satisfies both
monotonicity of evidence and evidence closure properties.
\end{lemma}
\begin{proof}
We start with the \textbf{monotonicity of evidence}. Assume
$R(m,m')$ and $t \in Pol$. If $m = \dag$, then, by Lemma
\ref{proven}, $\mathcal{E}(m,t) = \{ A \in Form \mid \vdash t\co A
\} \subseteq \mathcal{E}(m',t)$ for any $m' \in Tree$.

Assume, further, that $m$ is standard. Let $t \in Pol$ and $A \in
Form$ be such that $A \in \mathcal{E}(m,t)$. Then, for every $\xi
\in \overrightarrow{m}$, $t\co A \in end(\xi)$, and, by Lemma
\ref{theorems}.1, also $Kt\co A \in end(\xi)$. Therefore, by
$R(m,m')$, we get that, for every $\tau \in \overrightarrow{m'}$,
$Kt\co A \in end(\tau)$, so that, by \eqref{A7} and
maxiconsistency of every $end(\tau)$, also $t\co A \in end(\tau)$.
Therefore, $A \in \mathcal{E}(m',t)$, as desired.

We turn now to the \textbf{closure conditions}. We verify the
first two conditions, and the third one can be verified in a
similar way, restricting attention to $t$ rather than considering
both $s$ and $t$. Let $s, t \in Pol$. We need to consider two
cases:

\emph{Case 1}. $m = \dag$. If $A \in \mathcal{E}(m,s)$, then
$\vdash s\co A$. Therefore, by \eqref{A6}, we must also have
$\vdash (s + t)\co A$ so that $A \in \mathcal{E}(m,s + t)$.
Similarly, if $A \in \mathcal{E}(m,t)$, then also $A \in
\mathcal{E}(m,s + t)$ and the closure constraint (b) is verified.
If, on the other hand, it is true that for some $A, B \in Form$ we
have both $A \to B \in \mathcal{E}(m,s)$ and $A \in
\mathcal{E}(m,t)$, then, again, this means that both $\vdash s\co
A \to B$ and $\vdash t\co A$. By \eqref{A4}, it follows that
$\vdash s\times t\co B$ and, therefore, also $B \in
\mathcal{E}(m,s\times t)$, so that the closure condition (a) is
also verified.

\emph{Case 2}. $m \neq \dag$. If $A \in Form$ and  $A \in
\mathcal{E}(m,s)$, then, for every $\xi \in \overrightarrow{m}$,
$s\co A \in end(\xi)$, and, by \eqref{A6} and maxiconsistency of
every $end(\xi)$, we get that $s + t\co A \in end(\xi)$.
Therefore, $A \in \mathcal{E}(m,s + t)$. Similarly, if $A \in
\mathcal{E}(m,t)$, then $A \in \mathcal{E}(m,s + t)$ as well, and
closure condition (b) is verified.

On the other hand, if $A, B \in Form$ and we have both $A \to B
\in \mathcal{E}(m,s)$ and $A \in \mathcal{E}(m,t)$, then, for
every $\xi \in \overrightarrow{m}$, we have $t\co A, s\co(A \to B)
\in end(\xi)$. By \eqref{A4} and maxiconsistency of every
$end(\xi)$, we get that $s \times t\co B \in end(\xi)$, thus $B
\in \mathcal{E}(m,s \times t)$, and closure condition (a) is
verified.
\end{proof}

\subsection{$Act$ and $V$}

It only remains to define $Act$ and $V$ for our canonical model,
and we define them as follows:
\begin{itemize}
    \item $(m,h) \in V(p) \Leftrightarrow p \in end(m \sqcap h)$,
    for all $p \in Var$;
    \item $Act(\dag, h) = \emptyset$ for all $h \in Hist(\mathcal{M})$;
    \item $Act(m,h) = \{ t \in Pol \mid Et \in end(m \sqcap h)
    \}$, if $m \neq \dag$, $|m| = 0$ and $h \in H_m$;
    \item $Act(m,h) = \{ t \in Pol \mid Et \in \bigcap_{g \in H_m}end(m \sqcap g)
    \}$, if $m \neq \dag$, $|m| > 0$ and $h \in H_m$
\end{itemize}

We first draw some of the immediate consequences of the above
definitions:
\begin{lemma}\label{act-technical}
Assume that $m \in Tree \setminus \{ \dag \}$ and $t \in Pol$.
Then the following statements are true:
\begin{enumerate}
\item $Et \in \bigcap_{h \in H_m}(end(m \sqcap h)) \Leftrightarrow
t \in \bigcap_{h \in H_m}(Act(m,h))$;

\item If $|m| > 0$ and $h, h' \in H_m$, then $Act(m,h) =
Act(m,h')$;

\item If $h,h' \in H_m$ and $m \sqcap h = m \sqcap h'$, then
$Act(m,h) = Act(m,h')$.
\end{enumerate}
\end{lemma}
\begin{proof}
(Part 1). Let $g \in H_m$ be arbitrary. If $Et \in \bigcap_{h \in
H_m}(end(m \sqcap h))$, then $t \in Act(m,g)$ whatever the height
of $m$ is. Since $g$ was chosen arbitrarily, this means that $t
\in \bigcap_{h \in H_m}(Act(m,h))$. In the other direction, assume
that $t \in Act(m,g)$. Then, again irrespectively of the height,
$Et \in end(m \sqcap g)$. Therefore, if $t \in \bigcap_{h \in
H_m}(Act(m,h))$, then $Et \in \bigcap_{h \in H_m}(end(m \sqcap
h))$.

(Part 2). In the assumptions of this part, we get that:
$$
t \in Act(m,h) \Leftrightarrow Et \in \bigcap_{g \in H_m}(end(m
\sqcap g)) \Leftrightarrow t \in Act(m,h'),
$$
for an arbitrary $t \in Pol$.

(Part 3). We have to distinguish between two cases. If $|m| = 0$,
then, for an arbitrary $t \in Pol$, we get that:
$$
t \in Act(m,h) \Leftrightarrow Et \in end(m \sqcap h)
\Leftrightarrow Et \in end(m \sqcap h') \Leftrightarrow t \in
Act(m,h').
$$
On the other hand, if $|m| > 0$, then we are done by Part 2.
\end{proof}
 We now check the remaining semantic constraints on normal jstit models:

\begin{lemma}\label{act}
The canonical model, as defined above, satisfies the constraints
as to the expransion of presented proofs, no new proofs
guaranteed, presenting a new proof makes histories divide, and
epistemic transparency of presented proofs.
\end{lemma}
\begin{proof}
We consider the \textbf{expansion of presented proofs} first. Let
$m' \lhd m$ and let $h \in H_m$. If $m' = \dag$, then  we have
$Act(\dag, h) = \emptyset$, so that the expansion of presented
proofs holds. If $m' \neq \dag$, then $m$ is also standard.
Consider then $m' \sqcap h$ and $m \sqcap h$. Both these elements
must be in the basic sequence $]h[$, therefore, one of them must
be an initial segment of another. By $m' \lhd m$ we know that $m'
\sqcap h$ must be a proper initial segment of $m \sqcap h$. So we
may assume that $m' \sqcap h = (\Gamma_1,\ldots, \Gamma_k)$ and $m
\sqcap h = (\Gamma_1,\ldots, \Gamma_n)$ for some appropriate
$\Gamma_1,\ldots, \Gamma_n \subseteq Form$ and $n > k$. Now, if $t
\in Act(m',h)$, then $Et \in end(m' \sqcap h) = \Gamma_k$. Then,
since $(\Gamma_1,\ldots, \Gamma_n)$ is an element, we must have
$\Box Et \in \Gamma_n$. By definition of $\equiv$, it follows that
for every $\xi \in \overrightarrow{m}$ we must have that $Et \in
end(\xi)$. Now, if $g \in H_m$, then of course $m \sqcap g \in
\overrightarrow{m}$. Therefore, we get that $Et \in \bigcap_{g \in
H_{m}}end(m \sqcap g)$, whence, by Lemma \ref{act-technical}.1, $t
\in Act(m,h)$ immediately follows.

We consider next the \textbf{no new proofs guaranteed} constraint.
Let $m \in Tree$. If $m = \dag$, then $\bigcap_{h \in
H_m}(Act(m,h)) = \bigcup_{m' \lhd m, h \in H_m}(Act(m',h)) =
\emptyset$ and the constraint is trivially satisfied. Assume that
$m \neq \dag$. Then $m$ must be of the form $([(\Gamma_1,\ldots,
\Gamma_n)]_\equiv, k)$ for appropriate $\Gamma_1,\ldots, \Gamma_n
\subseteq Form$ and $k \in \omega$. Assume that $t \in \bigcap_{h
\in H_m}(Act(m,h))$. By Lemma \ref{act-technical}.1, we get then
that $Et \in \bigcap_{h \in H_{m}}end(m \sqcap h)$. Now, consider
$m' = ([(\Gamma_1,\ldots, \Gamma_n)]_\equiv, k+1)$. We clearly
have $m' \lhd m$, therefore, if $g \in H_m$, then also $g \in
H_{m'}$. In the other direction, if $g \in H_{m'}$, then, by
Corollary \ref{propertycorollary}, we get $g \in H_m$, so that the
fans of histories passing through $m$ and $m'$ are identical.
Further, we have $\overrightarrow{m} = \overrightarrow{m'}$, hence
it follows from Lemma \ref{property1} that $g \sqcap m = g \sqcap
m'$, whence $end(g \sqcap m) = end(g \sqcap m')$ for every $g \in
H_m = H_{m'}$, and, further, $\bigcap_{h \in H_{m}}end(m \sqcap h)
= \bigcap_{h \in H_{m'}}end(m' \sqcap h)$. Therefore, $Et \in
\bigcap_{h \in H_{m'}}end(m' \sqcap h)$ and it follows, by Lemma
\ref{act-technical}.1, that $t \in Act(m',h) \subseteq \bigcup_{m'
\lhd m, h \in H_m}(Act(m',h))$.

We turn next to the \textbf{presenting a new proof makes histories
divide} constraint. Consider an $m, m' \in Tree$ such that $m \lhd
m'$ and arbitrary $h, h' \in H_{m'}$. We immediately get then that
$h, h' \in H_{m}$.  If $m = \dag$, then the constraint is verified
trivially. If $m \neq \dag$, then we have two cases to consider:

\emph{Case 1}. $\overrightarrow{m} = \overrightarrow{m'}$ and $|m|
> |m'|$. Then we must have $|m| > 0$, and by Lemma \ref{act-technical}.2 it follows
that in this case for all $h, h' \in H_m$ we will have $Act(m, h)
= Act(m,h')$ so that the constraint is verified.

\emph{Case 2}. There is a $\xi \in \overrightarrow{m}$ such that
$\xi$ is a proper initial segment of every $\tau \in
\overrightarrow{m'}$. Consider then $m' \sqcap h$ and $m' \sqcap
h'$. These are elements in $\overrightarrow{m'}$, and hence $\xi$
is a proper initial segment of both $m' \sqcap h$ and $m' \sqcap
h'$. It follows that $m \sqcap h = m \sqcap h' = \xi$ whence, by
Lemma \ref{act-technical}.3, we immediately get $Act(m,h) =
Act(m,h')$.

It remains to check the \textbf{epistemic transparency of
presented proofs} constraint. Assume that $m, m' \in Tree$ are
such that $R(m,m')$. If we have $m = \dag$, then, by definition,
we must have $\bigcap_{h \in H_m}(Act(m,h)) = \emptyset$, and the
constraint is verified in a trivial way. If, on the other hand, $m
\neq \dag$, then, by $R(m,m')$, we must also have $m' \neq \dag$.
Assume that $t \in \bigcap_{h \in H_m}(Act(m,h))$. Then, by Lemma
\ref{act-technical}.1, we also have $Et \in \bigcap_{h \in
H_m}(end(m \sqcap h))$. Let $h \in H_m$ be arbitrary. We claim
that under these assumptions, we must have $\Box Et \in end(m
\sqcap h)$. Indeed, if $\Box Et \notin end(m \sqcap h)$, then
consider the following set $\Xi$ of formulas:
$$
\Xi = \{ \Box B\mid \Box B \in end(m \sqcap h) \} \cup \{ \neg Et
\}.
$$
We claim that $\Xi$ is consistent. Otherwise we would have
$$
\vdash (\Box B_1 \wedge\ldots\wedge \Box B_n) \to Et
$$
for some $\Box B_1,\ldots, \Box B_n \in end(m \sqcap h)$, and the
latter, by S5 reasoning for $\Box$, would mean that
$$
\vdash (\Box B_1 \wedge\ldots\wedge \Box B_n) \to \Box Et,
$$
whence, by Lemma \ref{elementaryconsistency} and maxiconsistency
of $end(m \sqcap h)$, $\Box Et \in end(m \sqcap h)$ would follow,
contrary to our hypothesis. But then we can extend $\Xi$ to a
maxiconsistent $\Delta$. Assume that $m \sqcap h =
(\Gamma_1,\ldots,\Gamma_k,\Gamma)$, so that $\Gamma = end(m \sqcap
h)$. We show that $(\Gamma_1,\ldots,\Gamma_k,\Delta) \in
\overrightarrow{m}$. We start by showing that
$(\Gamma_1,\ldots,\Gamma_k,\Delta)$ is an element. If $KB \in
\Gamma_k$, then, since $(\Gamma_1,\ldots,\Gamma_k,\Gamma)$ is an
element, it follows that $KB \in \Gamma$. By Lemma
\ref{theorems}.2 and maxiconsistency of $\Gamma$, we further get
that $\Box KB \in \Gamma$, whence $\Box KB \in \Delta$, and, by
\eqref{A1} and maxiconsistency of $\Delta$, $KB \in \Delta$.
Similarly, if $Es \in \Gamma_k$, then $\Box Es \in \Gamma$ and
further, $\Box Es \in \Delta$. Once
$(\Gamma_1,\ldots,\Gamma_k,\Delta)$ is thus shown to be an
element, $(\Gamma_1,\ldots,\Gamma_k,\Gamma) \equiv
\Gamma_1,\ldots,\Gamma_k,\Delta)$ follows immediately just by the
choice of $\Xi$ and the fact that $\Delta$ extends $\Xi$.
Therefore, $(\Gamma_1,\ldots,\Gamma_k,\Delta) \in
\overrightarrow{m}$. By Lemma \ref{sqcap} there is a $g \in H_m$
such that $m \sqcap g = (\Gamma_1,\ldots,\Gamma_k,\Delta)$. Then
$\Delta = end(m \sqcap g)$, but we also know that $\neg Et \in
\Delta$. Therefore, by maxiconsistency, $Et \notin \Delta = end(m
\sqcap g)$. But this is in contradiction with our assumption that
$Et \in \bigcap_{h \in H_m}(m \sqcap h)$.

The obtained contradiction shows that $\Box Et \in end(m \sqcap
h)$, and by maxiconsistency of $end(m \sqcap h)$ and \eqref{A9},
this means that also $K\Box Et \in end(m \sqcap h)$. It remains to
note that we have, of course $m = ([m \sqcap h]_\equiv, |m|)$,
whence by $R(m,m')$ we get that $K\Box Et \in \tau$ for every
$\tau \in \overrightarrow{m'}$. This means, by maxiconsistency of
every such $\tau$, \eqref{A1}, and \eqref{A7}, that $Et \in \tau$
for every $\tau \in \overrightarrow{m'}$. Note, further, that if
$g \in H_{m'}$, then $m' \sqcap g \in \overrightarrow{m'}$, so
that we have shown that $Et \in \bigcap_{g \in H_{m'}}(m' \sqcap
g)$, and hence, by Lemma \ref{act-technical}.1, also $t \in
\bigcap_{g \in H_{m'}}(Act(m',g))$, as desired.
\end{proof}

\subsection{The truth lemma}
It follows from Lemmas \ref{leq}--\ref{act}, that our
above-defined canonical model is in fact a normal unirelational
jstit model. Now we need to supply a truth lemma:
\begin{lemma}\label{truth}
Let $A \in Form$, let $m \in Tree \setminus \{ \dag \}$ be such
that $|m| = 0$, and let $h \in H_m$. Then:
$$
\mathcal{M},m,h \models A \Leftrightarrow A \in end(m \sqcap h).
$$
\end{lemma}
\begin{proof}
As is usual, we prove the lemma by induction on the construction
of $A$. The basis of induction with $A = p \in Var$ we have by
definition of $V$, whereas Boolean cases for the induction step
are trivial. We treat the modal cases:

\emph{Case 1}. $A = \Box B$. If $\Box B \in end(m \sqcap h)$, then
note that for every $h' \in H_m$ we must have $m \sqcap h' \in m$
so that $m \sqcap h' \equiv m \sqcap h$. By definition of $\equiv$
and the fact that $m \in Tree \setminus \{ \dag \}$, we must have
then $B \in end(m \sqcap h')$ for all $h' \in H_m$ and thus, by
induction hypothesis, we obtain that $\mathcal{M},m,h \models \Box
B$. If, on the other hand, $\Box B \notin end(m \sqcap h)$, then
let $m \sqcap h = (\Gamma_1,\ldots,\Gamma_k,\Gamma)$ so that
$end(m \sqcap h) = \Gamma$. Then the set
$$
\Xi = \{ \Box C \mid \Box C \in \Gamma \} \cup \{ \neg B \}
$$
must be consistent, since otherwise we would have
$$
\vdash (\Box C_1\wedge\ldots\wedge\Box C_n) \to B
$$
for some $\Box C_1,\ldots,\Box C_n \in \Gamma$, whence, since
$\Box$ is an S5-modality, we would get
$$
\vdash (\Box C_1\wedge\ldots\wedge\Box C_n) \to \Box B,
$$
which would mean that $\Box B \in \Gamma$, contrary to our
assumption. Therefore, $\Xi$ is consistent and we can extend $\Xi$
to a maxiconsistent $\Delta \subseteq Form$. Of course, in this
case $B \notin \Delta$. We now show that
$(\Gamma_1,\ldots,\Gamma_k,\Delta)$ is an element. Indeed, if for
any $C \in Form$ we have that $KC \in \Gamma_k$, then, since
$(\Gamma_1,\ldots,\Gamma_k,\Gamma)$ is an element, we will have
$KC \in \Gamma$, whence, by maxiconsistency of $\Gamma$ and
\eqref{A8}, $\Box KC \in \Gamma$, and since every boxed formula
from $\Gamma$ is also in $\Delta$, we get that $\Box KC \in
\Delta$, whence $KC \in \Delta$ by maxiconsistency of $\Delta$ and
S5 reasoning for $\Box$. Further, if we have $Et \in \Gamma_k$,
for some $t \in Pol$, then, since
$(\Gamma_1,\ldots,\Gamma_k,\Gamma)$ is an element, we will have
$\Box Et \in \Gamma$, and since every boxed formula from $\Gamma$
is also in $\Delta$, we get that $\Box Et \in \Delta$.

Once we know that $(\Gamma_1,\ldots,\Gamma_k,\Delta)$ is an
element, it follows by the choice of $\Xi$ and $\Delta$ that
$(\Gamma_1,\ldots,\Gamma_k,\Gamma) \equiv
(\Gamma_1,\ldots,\Gamma_k,\Delta)$. By Lemma \ref{sqcap}, for some
$h' \in H_m$ we will have $(\Gamma_1,\ldots,\Gamma_k,\Delta) = m
\cap h'$ and, therefore, $\Delta = end(m \sqcap h')$. Since $B
\notin \Delta$, it follows, by induction hypothesis, that
$\mathcal{M},m,h' \not\models B$, hence $\mathcal{M},m,h
\not\models \Box B$ as desired.

\emph{Case 2}. $A = [j]B$ for some $j \in Ag$. Then, if $[j]B \in
end(m \sqcap h)$, by definition of $Choice$ and the fact that both
$m \neq \dag$ and $|m| = 0$ we must have:
$$
Choice^m_j(h) = \{ h' \in H_m \mid
    (\forall C \in Form)([j]C \in end(h \sqcap m) \Rightarrow C \in end(h' \sqcap
    m))\}.
$$
Therefore, if $h' \in Choice^m_j(h)$, then we must have that $B
\in end(h' \sqcap m)$, and further, by induction hypothesis, that
$\mathcal{M},m,h' \models B$, so that we get $\mathcal{M},m,h
\models [j]B$. On the other hand, if $[j]B \notin end(m \sqcap
h)$, we again assume that $m \sqcap h =
(\Gamma_1,\ldots,\Gamma_k,\Gamma)$ so that $end(m \sqcap h) =
\Gamma$. Then the set
$$
\Xi = \{ [j]C \mid [j]C \in \Gamma \} \cup \{ \neg B \}
$$
must be consistent, since otherwise we would have
$$
\vdash ([j]C_1\wedge\ldots\wedge[j]C_n) \to B
$$
for some $[j]C_1,\ldots,[j]C_n \in \Gamma$, whence, since $[j]$ is
an S5-modality, we would get
$$
\vdash ([j]C_1\wedge\ldots\wedge[j]C_n) \to [j]B,
$$
which would mean that $[j]B \in \Gamma$, contrary to our
assumption. Therefore, $\Xi$ is consistent and we can extend $\Xi$
to a maxiconsistent $\Delta \subseteq Form$. Of course, in this
case $B \notin \Delta$. Arguing as in Case 1, we can show that
$(\Gamma_1,\ldots,\Gamma_k,\Delta)$ is an element.

Now, if $D \in Form$ is such that $\Box D \in \Gamma$, then, by
\eqref{A2} and maxiconsistency of $\Gamma$, we know that $[j]D \in
\Gamma$, so that also $[j]D \in \Delta$, and hence, by \eqref{A1}
and maxiconsistency of $\Delta$, $D \in \Delta$. We have thus
shown that:
 \begin{equation}\label{E:t1}
(\forall D \in Form)(\Box D \in \Gamma \Rightarrow D \in \Delta),
\end{equation}
 and it follows that $(\Gamma_1,\ldots,\Gamma_k,\Gamma) \equiv
(\Gamma_1,\ldots,\Gamma_k,\Delta)$ by definition of $\equiv$. By
Lemma \ref{sqcap}, for some $h' \in H_m$ we will have
$(\Gamma_1,\ldots,\Gamma_k,\Delta) = m \sqcap h'$ and, therefore,
$\Delta = end(m \sqcap h')$. Also, since $\Delta$ contains all the
$[j]$-modalized formulas from $\Gamma$, we know that for any such
$h'$ we will have $h' \in Choice^m_j(h)$. Since $B \notin \Delta$,
it follows, by induction hypothesis, that $\mathcal{M},m,h'
\not\models B$, hence $\mathcal{M},m,h \not\models [j]B$ as
desired.

\emph{Case 3}. $A = KB$. Assume that $KB \in end(m \sqcap h)$. We
clearly have then $m = ([(m \sqcap h)]_\equiv, 0)$. Hence, by
definition of $R$ and the fact that $m \neq \dag$ we must have for
every $m' \in Tree$:
$$
R(m,m') \Rightarrow (\forall \tau \in m')(\forall C \in Form)(KC
\in end(m \cap h) \Rightarrow KC \in end(\tau)).
$$
Therefore, if $R(m,m')$ and $h' \in H_{m'}$ is arbitrary, then, of
course, $(h' \sqcap m') \in m'$ so that $KB \in end(h' \sqcap
m')$, and, further, $B \in end(h' \sqcap m')$ by S4 reasoning for
$K$. Therefore, by induction hypothesis, we get that
$\mathcal{M},m',h' \models B$, whence $\mathcal{M},m,h \models
KB$. On the other hand, if $KB \notin end(m \sqcap h)$, then
consider the set
$$
\Xi = \{ KC \mid KC \in end(m \sqcap h) \} \cup \{ \neg\Box B \}.
$$
This set must be consistent, since otherwise we would have
$$
\vdash (KC_1\wedge\ldots\wedge KC_n) \to \Box B
$$
for some $KC_1,\ldots,KC_n \in \Gamma$, whence, since $K$ is an
S4-modality, we would get
$$
\vdash (KC_1\wedge\ldots\wedge KC_n) \to K\Box B,
$$
which would mean that $K\Box B \in end(m \sqcap h)$, hence, by
\eqref{A1}, \eqref{A7} and maxiconsistency of $end(m \sqcap h)$,
that $KB \in end(m \sqcap h)$, contrary to our assumption.
Therefore, $\Xi$ is consistent and we can extend $\Xi$ to a
maxiconsistent $\Delta \subseteq Form$. Of course, in this case
$\Box B \notin \Delta$. We will have then that $(\Delta)$ is an
element. So we set $m' = ([(\Delta)]_\equiv,0)$. Assume that
$(\Delta') \equiv (\Delta)$. Then every boxed formula from
$\Delta$ will be in $\Delta'$. In particular, whenever $KC \in
\Delta$, then also $\Box KC \in \Delta$ and thus $KC \in \Delta'$,
by \eqref{A1}, \eqref{A8}, and maxiconsistency of $\Delta$.
Therefore, whenever $KC \in end(m \sqcap h)$ and $\tau \in
\overrightarrow{m'} = [(\Delta)]_\equiv$, we have that $KC \in
end(\tau)$ so that we must have $R(m,m')$. On the other hand,
since $\Box B \notin \Delta$, then, by Case 1, there must be a
$\tau \in m'$ such that $B \notin end(\tau)$. But then, by Lemma
\ref{sqcap}, we can choose an $h' \in H_{m'}$ in such a way that
$\tau = m' \sqcap h'$, and we get that $B \notin end(m' \cap h')$.
Therefore, by induction hypothesis, we get $\mathcal{M},m',h'
\not\models B$. In view of the fact that also $R(m,m')$, this
means that $\mathcal{M},m,h \not\models KB$ as desired.

\emph{Case 4}. $A = t\co B$ for some $t \in Pol$. If $t\co B \in
end(m \sqcap h)$, then, by maxiconsistency of $end(m \sqcap h)$
and Lemma \ref{theorems}.1, we must have $\Box t\co B \in end(m
\sqcap h)$. Now, if $\xi \in \overrightarrow{m}$, then we must
have, of course $\xi \equiv m \sqcap h$, whence $t\co B \in
end(\xi)$. Therefore, we must have $B \in \mathcal{E}(m,t)$. Also,
by maxiconsistency of $end(m \sqcap h)$ and \eqref{A5}, we will
have $KB \in end(m \sqcap h)$. Therefore, by Case 3, we will have
that $\mathcal{M},m,h \models KB$ and further, by $B \in
\mathcal{E}(m,t)$, that $\mathcal{M},m,h \models t\co B$. On the
other hand, if $t\co B \notin end(m \sqcap h)$, then, since
clearly $m \sqcap h \in \overrightarrow{m}$, we must have $B
\notin \mathcal{E}(m,t)$, whence $\mathcal{M},m,h \not\models t\co
B$.

\emph{Case 5}. $A = Et$ for some $t \in Pol$. Then, given that  $m
\neq \dag$ and $|m| = 0$, we have, simply by definition of $Act$,
that:
$$
Et \in end(m \sqcap h) \Leftrightarrow t \in Act(m,h)
\Leftrightarrow \mathcal{M},m,h \models Et.
$$

This finishes the list of the modal induction cases at hand, and
thus the proof of our truth lemma is complete.
\end{proof}

\section{The main result}\label{main}

We are now in a position to prove Theorem \ref{completeness}. The
proof proceeds as follows. One direction of the theorem was proved
as Corollary \ref{c-soundness}. In the other direction, assume
that $\Gamma \subseteq Form$ is consistent. Then, by Lemma
\ref{elementaryconsistency}.1, $\Gamma$ can be extended to a
maxiconsistent $\Delta$. But then consider $\mathcal{M} = \langle
Tree, \unlhd, Choice, Act, R, \mathcal{E}, V\rangle$, the
canonical model defined in Section \ref{canonicalmodel}. It is
clear that $(\Delta)$ is an element, therefore $m =
([(\Delta)]_\equiv, 0) \in Tree$. By Lemma \ref{sqcap}, there is a
history $h \in H_m$ such that $(\Delta) = ([(\Delta)]_\equiv, 0)
\sqcap h$. For this $h$, we will also have $\Delta =
end(([(\Delta)]_\equiv, 0) \sqcap h)$. By Lemma \ref{truth}, we
therefore get that:
$$
\mathcal{M}, ([(\Delta)]_\equiv, 0), h \models \Delta \supseteq
\Gamma,
$$
and thus $\Gamma$ is shown to be satisfiable in a normal jstit
unirelational model, hence in a normal jstit model.

\textbf{Remark}. Note that the canonical model used in this proof
is universal in the sense that it satisfies every subset of $Form$
which is consistent relative to $\Sigma$.

As an obvious corollary of Theorem \ref{completeness} we get the
compactness property:

\begin{corollary}\label{compactness}
An arbitrary $\Gamma \subseteq Form$ is satisfiable in a normal
(unirelational) jstit model iff every finite $\Gamma_0 \subseteq
\Gamma$ is satisfiable in a normal (unirelational) jstit model.
\end{corollary}

The construction of the canonical model defined in Section
\ref{canonicalmodel} allows for a generalization. Let us call a
\emph{constant specification} any set $\mathcal{CS}$ such that:
\begin{itemize}
\item $\mathcal{CS} \subseteq \{ c_n\co\ldots c_1\co A\mid
c_1,\ldots,c_n \in PConst, A\textup{ an instance of
}\eqref{A0}-\eqref{A9}\}$;

\item Whenever $c_{n+1}\co c_n\co\ldots c_1\co A \in
\mathcal{CS}$, then also $c_n\co\ldots c_1\co A \in \mathcal{CS}$.
\end{itemize}
For a given constant specification, we can define the
corresponding inference rule $R_{\mathcal{CS}}$ as follows:
\begin{align}
\textup{If }c_n\co\ldots c_1\co A \in \mathcal{CS},\textup{ infer
} c_n\co\ldots c_1\co A.\label{RCS}\tag{\text{$R_{\mathcal{CS}}$}}
\end{align}
It is easy to see that the least constant specification will be
just $\emptyset$ and that $R3$ is in fact $R_{\mathsf{CS}}$ where
$\mathsf{CS}$ is the following constant specification:
$$
\{ c\co A\mid c \in PConst, A\textup{ an instance of
 } \eqref{A0}-\eqref{A9}\}.
$$
We note that Theorem \ref{completeness} is accordingly but a
particular instance, obtained by setting $\mathcal{CS}:=
\mathsf{CS}$, of the following more general theorem:
\begin{theorem}\label{completeness-general}
Let $\mathcal{CS}$ be a constant specification. Then an arbitrary
$\Gamma \subseteq Form$ is consistent relative to the axiomatic
system $\Sigma_{\mathcal{CS}} = \{ \eqref{A0}-\eqref{A9},
\eqref{R1}, \eqref{R2}, \eqref{R4}, \eqref{RCS}\}$ iff $\Gamma$ is
satisfiable in an (unirelational) jstit model satisfying the
following additional condition:
\begin{align*}
(\forall c \in PConst)(\forall m \in Tree)(\{ A \mid c\co A\in
\mathcal{CS} \} \subseteq \mathcal{E}(m,c)).
\end{align*}
\end{theorem}
We further note that the proof of this more general theorem can be
obtained from the proof of Theorem \ref{completeness} above simply
by replacing every reference to $\Sigma = \Sigma_{\mathsf{CS}}$ by
a reference to $\Sigma_{\mathcal{CS}}$. We end this section by the
observation that it follows from Theorem
\ref{completeness-general} that the axiomatization of the
validities over the whole unrestricted class of (unirelational)
jstit models is given by $\Sigma_{\emptyset} = \{
\eqref{A0}-\eqref{A9}, \eqref{R1}, \eqref{R2}, \eqref{R4}\}$.

\section{Conclusion}\label{conclusion}
Building up on an earlier work on jstit formalisms, we have
defined stit logic of justification announcements (JA-STIT) --- a
natural logic which combines justification logic with stit logic
to provide a natural environment for representing proving activity
of agents within a (somewhat idealized) finite community of
researchers. For this logic, we have defined the semantics
originally presented in \cite{OLWA}. The main import of this paper
is that JA-STIT admits of a strongly complete axiomatization
w.r.t. this semantics and that this axiomatization can be
straighforwardly accommodated to a wide range of possible constant
specifications.

The main result of the present paper also leads to a number of
natural questions which we hope to be able to answer in our future
publications. One problem is posed by the fact, established in
Proposition \ref{proposition}, that JA-STIT is expressive enough
to distinguish between the class of all jstit models and the class
of all models based on discrete time. This fact implies that our
axiomatization will no longer be complete once the time is assumed
to be discrete. However, jstit models based on discrete time form
a very natural subclass within the class of jstit models, and it
would be nice to find out how to axiomatize our logic over this
particular subclass.

Another problem for future research is finding a separate
axiomatization for the explicit fragment of basic jstit logic. It
was mentioned above that even though in JA-STIT one can retrieve
explicit proving modalities of this logic, the inverse reduction
does not seem to go through, so that in terms of expressive power
JA-STIT appears to be a proper extension of the explicit fragment
of basic jstit logic. A natural further move would be then to find
a separate axiomatization for the explicit fragment of basic jstit
logic and compare it to the axiomatization presented in this
paper. Yet another natural, although by no means trivial, further
move would be to take on board also the implicit version $EA$ of
$Et$-modality and axiomatize the full logic of $E$-notions.

\section{Acknowledgements}
To be inserted.

}

\end{document}